\documentclass[11pt, oneside]{amsart}
\usepackage[text={5.58in,8.5in},centering,letterpaper,dvips]{geometry}
\usepackage[dvipsnames]{xcolor}
\usepackage{graphicx}
\usepackage[utf8]{inputenc}
\usepackage{subcaption}
\usepackage{amsfonts}
\usepackage{epsf}
\usepackage{amssymb}
\usepackage{amsmath}
\usepackage{amscd}
\usepackage{tikz}
\usepackage{pdfpages}
\usepackage{fancyhdr}
\usepackage{setspace}
\usepackage{hyperref}
\usepackage[all]{xy}
\usetikzlibrary{matrix}
\usepackage{verbatim}
\usepackage{enumerate}

\theoremstyle{theorem}
\newtheorem{theorem}{Theorem}[section]
\newtheorem{proposition}[theorem]{Proposition}
\newtheorem{lemma}[theorem]{Lemma}
\newtheorem{question}[theorem]{Question}

\newtheorem{conjecture}[theorem]{Conjecture}

\theoremstyle{definition}

\newtheorem{remark}[theorem]{Remark}

\newcommand{\Z}{\mathbb{Z}}
\newcommand{\C}{\mathbb{C}}

\newcommand{\R}{\mathbb{R}}

\newcommand{\Tt}{\mathcal T}

\newcommand{\wh}[1]{\widehat{#1}}

\newcommand{\A}{\alpha}
\newcommand{\n}{\beta}

\newcommand{\pd}{\partial}
\newcommand{\lk}{\ell k}
\newcommand{\emp}{\emptyset}
\newcommand{\X}{\times}

\newcommand{\be}{\begin{enumerate}}
\newcommand{\ee}{\end{enumerate}}
\newcommand{\eps}{\varepsilon}
\newcommand{\Ph}{\tilde p}

\makeatletter
\def\@seccntformat#1{%
  \protect\textup{\protect\@secnumfont
    \ifnum\pdfstrcmp{subsection}{#1}=0 \bfseries\fi
    \csname the#1\endcsname
    \protect\@secnumpunct
  }%
}  
\makeatother

\makeatletter
\newtheorem*{rep@theorem}{\rep@title}
\newcommand{\newreptheorem}[2]{%
\newenvironment{rep#1}[1]{%
 \def\rep@title{#2~\ref{##1}}%
 \begin{rep@theorem}}%
 {\end{rep@theorem}}}
\makeatother

\newreptheorem{theorem}{Theorem}
\newreptheorem{lemma}{Lemma}
\newreptheorem{question}{Question}
\newreptheorem{corollary}{Corollary}
\newreptheorem{proposition}{Proposition}

\begin{document}

\rhead{\thepage}
\lhead{\author}
\thispagestyle{empty}


\raggedbottom
\pagenumbering{arabic}
\setcounter{section}{0}


\title{Equivalent characterizations of handle-ribbon knots}

\author{Maggie Miller}
\address{Department of Mathematics, Princeton University, Princeton, NJ 08540}
\email{maggiem@math.princeton.edu}
\urladdr{https://web.math.princeton.edu/~maggiem}

\author{Alexander Zupan}
\address{Department of Mathematics, University of Nebraska-Lincoln, Lincoln, NE 68588}
\email{zupan@unl.edu}
\urladdr{http://www.math.unl.edu/~azupan2}

\begin{abstract}
The stable Kauffman conjecture posits that a knot in $S^3$ is slice if and only if it admits a slice derivative. We prove a related statement:  A knot is handle-ribbon (also called strongly homotopy-ribbon) in a homotopy 4-ball $B$ if and only if it admits an R-link derivative; i.e. an $n$-component derivative $L$ with the property that zero-framed surgery on $L$ yields $\#^n(S^1\times S^2)$.  We also show that $K$ bounds a handle-ribbon disk $D \subset B$ if and only if the 3-manifold obtained by zero-surgery on $K$ admits a singular fibration that extends over handlebodies in $B \setminus D$, generalizing a classical theorem of Casson and Gordon to the non-fibered case for handle-ribbon knots.
\end{abstract}

\maketitle

\section{Introduction}

One of the most well-known open problems in knot theory is the slice-ribbon conjecture of Fox, which proposes that every knot $K \subset S^3$ that bounds a smooth disk in $B^4$ also bounds an immersed ribbon disk in $S^3$~\cite{fox}.  In other words, if $K$ is slice, then $K$ is ribbon.  In this paper, we focus on characterizing sliceness, ribbonness, and an intermediate condition using \emph{derivative links}.  For a knot $K \subset S^3$ and genus $g$ Seifert surface $F$ for $S$, a \emph{derivative} $L$ for $K$ in $F$ is a $g$-component link such that $L \subset F$, $F - L$ is a connected planar surface, and $\lk(L_i,L^+_j)=0$ for all $i,j$, where $L^+_j$ is a parallel copy of $L_j$ pushed off of $F$.  

The following proposition is well-known; see~\cite{CD} for a proof.

\begin{proposition}\label{ribbon}
A knot $K \subset S^3$ is ribbon if and only if $K$ has an unlink derivative $U$.
\end{proposition}


A $g$-component link $L \subset S^3$ is \emph{slice} if $L$ bounds a collection of $g$ pairwise disjoint smooth disks in the 4-ball.  Regarding sliceness, Cochran and Davis made the following conjecture:

\begin{conjecture}\cite[stable Kauffman conjecture]{CD}\label{slice}
A knot $K \subset S^3$ is slice if and only if $K$ has a slice derivative $L$.
\end{conjecture}

In this work, we examine an intermediate family of knots.  A knot $K \subset S^3$ is said to be \emph{handle-ribbon} if $K$ bounds a disk $D$ in a homotopy 4-ball $B$ such that the exterior of $D$ can be built without 4-dimensional 3-handles.  An $n$-component link $L \subset S^3$ is called an \emph{R-link} if the manifold obtained by 0-surgery on each component of $L$ is $\#^n (S^1 \X S^2)$.  In parallel with Proposition~\ref{ribbon} and Conjecture~\ref{slice}, we prove

\begin{theorem}\label{hr}
A knot $K \subset S^3$ is handle-ribbon in a homotopy 4-ball if and only if $K$ has an R-link derivative.
\end{theorem}

Since every R-link is slice, it follows that handle-ribbon knots satisfy the stable Kauffman conjecture. We note that one direction of Conjecture~\ref{slice} is straightforward:  If $K$ has a slice derivative $L$ in a surface $F$, then a slice disk for $K$ is obtained by taking the union of $F - L$ and copies of the disks in $B^4$ bounded by $L$ (i.e. a slice disk is obtained by compressing $F$ along slice disks bounded by $L$).  The reverse direction remains open, although we should note that Cochran and Davis disproved the Kauffman conjecture, which posited that {\emph{every}} (genus one) Seifert surface for a slice knot contains a slice derivative. To the contrary, they exhibited a genus one knot $K$ without this property (although their example does have a genus two surface with such a derivative) ~\cite{CD}.

The reader may observe that our notion of a \emph{handle-ribbon knot} is the same as a \emph{strongly homotopy-ribbon knot} appearing elsewhere in the literature~\cite{Meier-Larson,Miller-Zemke,HKP}.  We feel that \emph{handle-ribbon knot} is more accurately descriptive.  We elaborate further in Remark~\ref{strongly} in Section~\ref{sec:prelim} below.

Our work also generalizes a classical theorem of Casson and Gordon about homotopy-ribbon knots, a condition slightly weaker than being handle-ribbon (see Section~\ref{sec:prelim} for relevant definitions and a detailed discussion).  They proved

\begin{theorem}\cite{CG}
A fibered knot $K \subset S^3$ is homotopy-ribbon in a homotopy 4-ball if and only if the fibration of the 0-surgery on $K$ extends over handlebodies.
\end{theorem}

Using Theorem~\ref{hr}, we can make an analogous statement for handle-ribbon knots, dropping the condition that $K$ must be fibered.

\begin{theorem}\label{extensions}
A knot $K \subset S^3$ is handle-ribbon in a homotopy 4-ball if and only if there exists a singular fibration of the 0-surgery on $K$ that extends over handlebodies.
\end{theorem}

We offer two precise versions of Theorem~\ref{extensions}; these assertions appear as Theorems~
\ref{fibration1} and~\ref{fibration2}.  Theorem~\ref{fibration1} extends a singular fibration $p$ from the 0-surgery on $K$, denoted $S^3_0(K)$, to a generic map (or Morse-2 function) $P$ from a 4-manifold $X$ to the annulus $S^1 \X I$ which maps $\pd X$ to $S^1 \X \{0\}$ and restricts to $p$ on $\pd X = S^3_0(K)$.  Alternatively, Theorem~\ref{fibration2} uses tools developed by the first author in~\cite{maggie} to produce a circular Morse function $\Ph:X \rightarrow S^1$ such that $\Ph|_{\pd X} = p$.  In either case, regular fibers of $p$ are capped off by handlebodies in $X$.  The proofs of Theorems~\ref{fibration1} and ~\ref{fibration2} are quite different, so we have included both.  The existence of either type of extension is equivalent, since both are equivalent to $K$ being handle-ribbon in a homotopy 4-ball, although the definition via circular Morse functions is less rigid than the definition via generic maps.  This flexibility may give a strategy for finding singular fibrations with smaller genus fibers that admit extensions (see Remark~\ref{thinremark}).

The plan of the paper is as follows:  In Section~\ref{sec:prelim}, we offer definitions and discuss handle-ribbon knots and homotopy-ribbon knots in greater detail.  In Section~\ref{sec:rlink}, we prove Theorem~\ref{hr}, and in Section~\ref{sec:extend}, we prove Theorem~\ref{fibration1}.  In Section~\ref{sec:alternate}, we adapt techniques from work of the first author~\cite{maggie} to prove Theorem~\ref{fibration2}.  Finally, in Section~\ref{sec:trisections}, we demonstrate that the extensions of singular fibrations over handlebodies by generics maps give rise to natural trisections, generalizing work of Jeffrey Meier and the second author~\cite{MZDehn,FHRkS}.

\subsection*{Acknowledgements}

A portion of this work was completed while we were guests at the Max Planck Institute for Mathematics in Bonn, and we are extremely grateful to MPIM for its support and hospitality.  We also thank Jeffrey Meier for helpful conversations and for asking about natural trisections of extensions of singular fibrations.  The first author is supported by NSF grant DGE-1656466. The second author is supported by NSF grant DMS-1664578.

\section{Preliminaries}\label{sec:prelim}
All work is in the smooth category, where $S^n$ and $B^n$ denote the standard smooth $n$-sphere and $n$-ball, respectively.  Manifolds are orientable unless otherwise noted.  If $Y$ is a submanifold of $X$, let $X \setminus Y = X - \eta(Y)$, where $\eta(\cdot)$ represents an open regular neighborhood.  A closed 3-manifold $Y$ is obtained by \emph{Dehn surgery} on a knot $K \subset S^3$ with slope $a/b$ if $Y$ is constructed by gluing a solid torus $V$ to $S^3 \setminus K$ via a diffeomorphism of their boundaries that maps a meridian of $V$ to the $a/b$ curve on $\pd (S^3 \setminus K)$ in preferred coordinates.  In this case, $Y$ is denoted $Y = S^3_{a/b}(K)$.  The \emph{dual} $K^*$ to $K$ is the core of the surgery solid torus $V \subset S^3_{a/b}(K)$.  These concepts and notation extend to $n$-component links $L \subset S^3$, where the boundary slope $a/b$ is replaced with an $n$-tuple of boundary slopes corresponding to the $n$ boundary components of $S^3 \setminus L$.  A closely related idea is a \emph{knot trace} or \emph{link trace}:  Given a boundary slope $a/b$ and a knot $K \subset S^3$, the \emph{knot trace} $X_{a/b}(K)$ is defined to be the compact 4-manifold obtained by attaching a 4-dimensional 2-handle to $B^4$ along $K$ with framing $a/b$.  The key relationship here is that $\pd(X_{a/b}(K)) = S^3_{a/b}(K)$.  A relative handle decomposition of $X_{a/b}(K)$ is obtained by attaching a 2-handle to $K^* \subset S^3_{a/b}(K)$ and capping off the resulting $S^3$ boundary component with a 4-handle.  Link traces are defined similarly.  We also use the notion of relative traces:  The \emph{relative trace} $B_{a/b}(K)$ is obtained by attaching a 2-handle to $K \X \{1\} \subset S^3 \X I$ with framing $a/b$.  Thus, the trace $X_{a/b}(K)$ may be obtained by capping off the $S^3$ boundary component of the relative trace $B_{a/b}(K)$ with a 4-ball.

A collection of links with interesting traces is the family of \emph{R-links}.  Recall that an $n$-component link $L \subset S^3$ is an R-link if $S^3_{\vec 0}(L)$ (the manifold obtained by performing $0$-surgery on each component of $L$) is the 3-manifold $\#^n(S^1 \X S^2)$.  This nomenclature arises from Property R conjecture (proved by Gabai~\cite{gabai}), which asserts that the only 1-component R-link is the unknot.  The generalized Property R conjecture (GPRC, Kirby Problem 1.82~\cite{kirby}) proposes that every R-link is handleslide equivalent to an $n$-component unlink.  The GPRC is discussed in great detail in~\cite{GST}.

The study of R-links is closely related to the smooth 4-dimensional Poincar\'e conjecture (S4PC):  If $L$ is an $n$-component R-link, then $L$ gives rise to a closed 4-manifold $X_L$ built from a 0-handle, $n$ 2-handles, $n$ 3-handles, and a 4-handle, obtained by capping off the link trace $X_{\vec 0}(L)$ (made up of a 0-handle and $n$ 2-handles) with $n$ 3-handles and a 4-handle, where the condition $S^3_{\vec 0}(L) = \#^n(S^1 \X S^2)$ implies that the capping off is possible.  Moreover, this capping is unique (up to diffeomorphism) by Laudenbach-Poenaru~\cite{LP}, and thus $L$ completely determines $X_L$ up to diffeomorphism.  Note that $X_L$ is simply-connected (since it can be built without 1-handles), and $\chi(X_L) = 2$, so that $X_L$ is a homotopy 4-sphere.  Conversely, if $X$ is any homotopy 4-sphere with a decomposition with no 1-handles, then $X$ gives rise to an R-link $L$ (the attaching link for the 2-handles in $X$) such that $X = X_L$.

\subsection{Between ribbonness and sliceness}

Consider $S^3$ as the boundary of a homotopy 4-ball $B$.  If $L \subset S^3$ bounds a collection $D$ of pairwise disjoint, properly embedded disks in $B$, we will consider various restrictions on the disks $D$.  First we state a standard lemma.  For a proof, see Lemma 2.1 of~\cite{FHRkS}.

\begin{lemma}
If $(S^3,L) = \pd (B,D)$, then $\pd(B \setminus D) = S^3_{\vec 0}(L)$.
\end{lemma}

Suppose that $(S^3,L) = \pd (B^4,D)$.  If the restriction of the radial Morse function $h$ on $B^4$ to $D$ is Morse and contains only saddles and minima (but no maxima), then the disks $D$ and link $L$ are called \emph{ribbon}.  In a natural construction (see~\cite[Section 6.2]{GS}, for example), the critical points of $h|_D$ for a ribbon disk $D$ can be used to give a handle decomposition of $B^4 \setminus D$ with a 0-handle, 1-handles, and 2-handles (but no 3-handles).  This construction motivates the next definition:  If $(S^3,L) = \pd(B,D)$, where $B$ is any homotopy 4-ball, such that $B \setminus D$ has a handle decomposition without 3-handles, then the disks $D$ and link $L$ are called \emph{handle-ribbon in $B$}.  (Handle-ribbon knots are also called \emph{strongly homotopy-ribbon} elsewhere; see Remark~\ref{strongly} below.)  Turning this handle decomposition upside down yields a relative handle decomposition of $B \setminus D$ obtained by attaching 2-, 3-, and 4-handles to $S^3_{\vec 0}(L)$, and thus the map $i_*:\pi_1(S^3 \setminus L) \rightarrow \pi_1(B \setminus D)$ induced by inclusion must be surjective, since the inclusion factors through the inclusion $S^3 \setminus L \hookrightarrow S^3_{\vec 0}(L)$.  Thus, if $(S^3,L) = \pd (B,D)$, where $i_*:\pi_1(S^3 \setminus L) \rightarrow \pi_1(B \setminus D)$ is surjective, then the disks $D$ and the link $L$ are called \emph{homotopy-ribbon in $B$}.  Finally, for any $(S^3,L) = \pd (B,D)$ without restrictions, the disks $D$ and link $L$ are called \emph{slice in $B$}.  

In summary
\begin{equation}\{\text{ribbon links} \} \subset \{\text{handle-ribbon links}\} \subset \{\text{homotopy-ribbon links}\} \subset \{ \text{slice links}\}.\label{sliceribbon}\end{equation}
However, none of these containments is known to be strict. The slice-ribbon conjecture posits that when only considering $B^4$, all of these containments are equivalences. Thus, affirmative answers to the slice-ribbon conjecture and the relative 4D Poincar\'{e} conjecture (that every homotopy 4-ball is diffeomorphic to $B^4$) would together imply all containments in~\eqref{sliceribbon} are equivalences.

 Note that any R-link $L$ can be used to build a homotopy 4-ball $B_L$ by removing the 0-handle of $X_L$, where the cores of the 2-handles attached along $L$ become handle-ribbon disks for $L$ in $B_L$.  Using this principle as a guide, Gompf, Scharlemann, and Thompson constructed their famous potential counterexamples to the slice-ribbon conjecture in~\cite{GST}.

\begin{remark}\label{strongly}
Our presented definition of a \emph{homotopy-ribbon knot} agrees with that of Casson-Gordon~\cite{CG} but differs from that of Cochran~\cite{cochran}.  To unify these two definitions, Meier and Larson renamed Cochran's alternative to be a \emph{strongly homotopy-ribbon knot}~\cite{Meier-Larson}, which has been used by several other authors as well~\cite{Miller-Zemke,HKP}.  In this paper, we have decided to use \emph{handle-ribbon knot} in the place of \emph{strongly homotopy-ribbon knot}, since we judge ``handle-ribbon" to be more clearly descriptive.  Whereas a homotopy-ribbon knot bounds a disk satisfying the same homotopy-theoretic condition as a ribbon disk, a handle-ribbon knot bounds a disk whose complement admits a handle decomposition resembling that of a ribbon disk.
\end{remark}

\subsection{Stable equivalence of R-links}

Suppose that $L$ is a framed link, with components $L_1, L_2 \subset L$.  Let $\A$ be a framed arc connecting $L_1$ and $L_2$ and with $\mathring{\A}$ disjoint from $L$.  Then $L_1 \cup L_2 \cup \A$ has a framed pair of pants neighborhood $N \subset S^3$, where two boundary components of $N$ are isotopic to $L_1$ and $L_2$, and the third, call it $L_1'$, is said to be related to $L_1$ and $L_2$ by a \emph{handleslide}.  If $L'$ is the link obtained by replacing $L_1$ with $L_1'$, we say $L$ and $L'$ are \emph{handleslide equivalent}, and it is well-known that if $L$ is a (zero-framed) R-link, then $L'$ is also an R-link, with $X_L' = X_L$.

Additionally, if $L$ is an R-link and $U$ is an unlink, then the split union $L \sqcup U$ is also an R-link, with $X_{L \sqcup U} = X_L$.  From a 4-dimensional perspective, the relative handle decomposition of $X_{L \sqcup U}$ is obtained by adding $|U|$ pairs of canceling 2- and 3-handles to the handle decomposition of $X_L$.  Thus, we say that two R-links $L$ and $L'$ are \emph{stably equivalent} if there are unlinks $U$ and $U'$ such that $L \sqcup U$ and $L' \sqcup U'$ are handleslide equivalent.  In this case, we again have $X_{L'} = X_L$, where the relative handle decompositions are related by handleslides and by adding/deleting canceling pairs of 2- and 3-handles.  As an example, consider the following lemma, which we use in Section~\ref{sec:rlink}.

\begin{lemma}\label{kremove}
Suppose that $K$ has a derivative $L$ such that $K \cup L$ is an R-link.  Then $L$ is an R-link, and $L$ and $K \cup L$ are stably equivalent.
\end{lemma}

\begin{proof}
Suppose $L \subset F$, where $F$ is a Seifert surface for $K$.  Since $F \setminus L$ is planar, there is a sequence of handleslides in $F$ of $K$ over components of $L$ converting $K$ to a trivial curve in $F \setminus L$, and thus $K \cup L$ and $L$ are stably equivalent.
\end{proof}

In order to better understand handleslide equivalence, we will repeatedly use the next lemma, the content of which is contained in the proof of Proposition 3.2 from~\cite{GST}.

\begin{lemma}\label{isotop}
Let $K$ and $J$ be disjoint 0-framed links in $S^3$. If $K$ is isotopic in $S^3_{\vec 0}(J)$ to another link $K'$ disjoint from the duals $J^*$, then $K \cup J$ and $K' \cup J$ are handleslide equivalent in $S^3$.
\end{lemma}

\begin{proof}
The lemma follows immediately from the observation that any isotopy of $K$ in $S^3_{\vec 0}(J)$ that passes a strand of $K$ through a dual in $J^*$ can be realized as a move in $S^3_{\vec 0} \setminus J^*$ by banding $K$ to a meridian of $J^*$.  But $S^3_{\vec 0} \setminus J^* = S^3 \setminus J$, where a meridian of $J^*$ in the former is a 0-framed pushoff in the latter, so this banding is precisely a handleslide over a component of $J$ in $S^3$.
\end{proof}

\subsection{The characterization of Casson and Gordon}

In a classical and celebrated result, Casson and Gordon provided an alternate characterization for fibered knots that are homotopy-ribbon in a homotopy 4-ball.  A knot $K$ in $S^3$ is \emph{fibered} if $K$ has a Seifert surface $F$ such that $S^3 \setminus K$ is the mapping torus $F \X_{\varphi} S^1$, where $\varphi: F \rightarrow F$ is a diffeomorphism such that $\varphi|_{\pd F}$ is the identity.  By capping off each fiber of $F \X_{\varphi} S^1$ with a disk, we can extend the fibration of $S^3 \setminus K$ over $S^3_0(K)$, where $\wh F$ is the (closed) capped off Seifert surface, $\wh \varphi:\wh F \rightarrow \wh F$ is the natural extension of $\varphi$ to $\wh F$, and $S^3_0(K) = \wh F \X_{\wh \varphi} S^1$.  The map $\wh \varphi$ is called the \emph{closed monodromy} associated to the fibered knot $K$.  If there exists a handlebody $H$ with $\wh F = \pd H$ and a diffeomorphism $\Phi:H \rightarrow H$ such that $\Phi|_{\pd H} = \wh \varphi$, we say that $\wh \varphi$ \emph{extends over $H$}.  With these definitions, we can state Casson-Gordon's result:

\begin{theorem}~\cite{CG}\label{cassongordon}
A fibered knot $K \subset S^3$ is homotopy-ribbon in a homotopy 4-ball $B$ if and only if the closed monodromy associated to $K$ extends over a handlebody $H$.
\end{theorem}

In this case, there is a homotopy 4-ball $B'$ (possibly different from $B$) containing a homotopy-ribbon disk $D$ such that $B' \setminus D = H \X_{\Phi} S^1$, so that each fiber $\wh F$ in the fibration of $\pd(B' \setminus D) = S^3_0(K) = \wh F \X_{\wh \varphi} S^1$ is capped off with a handlebody fiber of $H \X_{\Phi} S^1$.  For this reason, we say that the fibration of $S^3_0(K)$ \emph{extends over handlebodies}.

In the Sections~\ref{sec:extend} and~\ref{sec:alternate}, we discuss how a singular fibration can extend over handlebodies in order to prove an analogue of Casson-Gordon's theorem in the case that $K$ is not fibered.

\section{R-link derivatives}\label{sec:rlink}

In this section, we prove Theorem~\ref{hr}.  First, we offer a lemma connecting handle-ribbon knots and R-links.

\begin{lemma}\label{rhandle}
A knot $K \subset S^3$ is handle-ribbon in a homotopy 4-ball if and only if $K$ is a component of some R-link $J$.
\end{lemma}

\begin{proof}
Suppose $K$ is handle-ribbon, and let $D$ be a handle-ribbon disk for $K$ in a homotopy 4-ball $B$.  Then $B \setminus D$ has a relative handle decomposition without 1-handles.  Let $J' \subset \pd (B \setminus D) = S^3_0(K)$ be the attaching link for the 2-handles.  After an isotopy, we can assume that $J' \subset S^3 \setminus K$, and thus we can attach a 0-framed 2-handle along $K$, followed by the relative handle decomposition of $B \setminus D$ to obtain a relative handle decomposition of $B$ without 1-handles such that $J' \cup K$ is the attaching link for the 2-handles.  Since $B$ is a homotopy 4-ball, this handle decomposition has $n$ 2-handles and $n$ 3-handles for some $n$, from which it follows that $J' \cup K$ is an R-link.

On the other hand, if $K$ is a component of an R-link $J$, then the homotopy 4-ball $B_J$ has a relative handle decomposition without 1-handles.  It follows that for a core $D$ of the 2-handle attached along $K$ in $B_J$, the complement $B_J \setminus D$ also has a relative handle decomposition without 1-handles, so that $D$ is a handle-ribbon disk for $K$ in $B_J$.
\end{proof}

Recall the definition of a derivative link from the introduction.  Let $K$ be a knot in $S^3$ with Seifert surface $F$, and suppose that $J \subset F$ is a link such that $\lk(J_i,J_j^+) = 0$ for all $i,j$ and no two components of $J$ are homotopic in $F$.  This definition mirrors the definition of a derivative link, except that we do not require $F \setminus J$ to be a connected planar surface.  The next lemma gives a procedure for converting $J$ to a derivative link in the case that $K \cup J$ is an R-link; thus we call such a link $J$ a \emph{partial derivative} of $K$.

\begin{proposition}\label{meat}
Suppose that $K$ has a partial derivative $J$ in $F$ such that $K \cup J$ is an R-link.  Then $F$ contains an R-link derivative $L$ for $K$, where $L$ is stably equivalent to $K \cup J$.
\end{proposition}

\begin{proof}
Let $\mathcal{J}$ be the collection of partial derivatives $J' \subset F$ such that $K \cup J'$ is stably equivalent to $K \cup J$.  By assumption $\mathcal{J}$ is nonempty.  Let $J'$ be an element of $\mathcal{J}$ with the greatest number of link components.  We claim that $F \setminus J'$ is a union of planar surfaces.  If not, then $F \setminus J'$ has a component $F'$ with genus at least one and nonempty boundary.  In the 3-manifold $S^3_{\vec 0}(K \cup J') = \#^{|K\cup J'|}(S^1 \X S^2)$, the surface $F'$ can be capped off with copies of meridians $\{D_i\}$ of the surgery solid tori corresponding to $K \cup J'$ to get a closed surface $\wh F'$.  Since every surface of positive genus in $\#^{|K\cup J'|}(S^1 \X S^2)$ is compressible, there exists a compressing disk $D$ for $\wh F'$ with boundary $C$.

After isotopy, we may assume that $C$ is disjoint from the disks $\{D_i\}$ used to cap off $F'$, and thus $C \subset S^3 \setminus (K \cup J')$.  Let $J'' = J' \cup C$, where $C$ is framed by the disk $D$.  Since $C$ is isotopic to an unknotted curve in $S^3_{\vec 0}(K \cup J')$, it follows from Lemma~\ref{isotop} that $C$ can be handleslid over $K \cup J'$ in $S^3$ to become unknotted and unlinked, and thus $K \cup J''$ is stably equivalent to $K \cup J'$.  In addition, $J'' \subset F$, where each component has surface framing zero and $C$ is not homotopic to any component of $J'$, since $C$ is an essential curve in $F'$.  But this implies that $J'' \in \mathcal{J}$, so that $J'$ is not maximal, a contradiction.  We conclude that a maximal element $J'$ cuts $F$ into planar components (in fact, pairs of pants, although this observation is not necessary for the proof).

Finally, to obtain a derivative from $J'$, we can handleslide components of $J'$ over each other in $F$ to get the union of a link $L$ and a split unlink $U$, where $F \setminus L$ is connected and planar.  It follows that $K \cup L$ is an R-link, and by Lemma~\ref{kremove}, $L$ is an R-link derivative which is stably equivalent to $K \cup L$ and thus $K \cup J$, as desired.
\end{proof}

\begin{proof}[Proof of Theorem~\ref{hr}]
We prove the easier direction first.  Suppose $L$ is an R-link derivative for $K$ contained in a Seifert surface $F$ for $K$.  As in Lemma~\ref{kremove}, a sequence of handleslides in $F$ of $K$ over the components of $L$ converts $K$ into an unknotted, unlinked component $U$.  It follows that $L \cup U$ is an R-link, and thus $L \cup K$ is an R-link as well.  By Lemma~\ref{rhandle}, $K$ is handle-ribbon.

For the more difficult direction, suppose that $K$ is handle-ribbon, so that there is an R-link $K \cup J$ by Lemma~\ref{rhandle}.  Let $F'$ be any Seifert surface for $K$, chosen to meet $J$ minimally.  We claim that $F' \cap J = \emp$.  If not, then there is a component $J_i$ of $J$ such that $J_i \cap F' \neq \emp$, and using the fact that $\text{lk}(K,J_i) = 0$, we have that $J_i$ meets $F'$ in points of opposite orientation.  It follows that there is an arc $\alpha \subset J_i$ with both endpoints on the same side of $F'$ and interior disjoint from $F'$, so that the result $F''$ of tubing $F'$ along $\alpha$ is a Seifert surface for $K$ such that $|F'' \cap K| < |F' \cap K|$, a contradiction. (See Figure~\ref{fig:tubing}.)  We conclude that $F' \cap K = \emp$.

\begin{figure}[h!]
	\centering
\begingroup%
  \makeatletter%
  \providecommand\color[2][]{%
    \errmessage{(Inkscape) Color is used for the text in Inkscape, but the package 'color.sty' is not loaded}%
    \renewcommand\color[2][]{}%
  }%
  \providecommand\transparent[1]{%
    \errmessage{(Inkscape) Transparency is used (non-zero) for the text in Inkscape, but the package 'transparent.sty' is not loaded}%
    \renewcommand\transparent[1]{}%
  }%
  \providecommand\rotatebox[2]{#2}%
  \newcommand*\fsize{\dimexpr\f@size pt\relax}%
  \newcommand*\lineheight[1]{\fontsize{\fsize}{#1\fsize}\selectfont}%
  \ifx\svgwidth\undefined%
    \setlength{\unitlength}{181.99023389bp}%
    \ifx\svgscale\undefined%
      \relax%
    \else%
      \setlength{\unitlength}{\unitlength * \real{\svgscale}}%
    \fi%
  \else%
    \setlength{\unitlength}{\svgwidth}%
  \fi%
  \global\let\svgwidth\undefined%
  \global\let\svgscale\undefined%
  \makeatother%
  \begin{picture}(1,0.33659218)%
    \lineheight{1}%
    \setlength\tabcolsep{0pt}%
    \put(0,0){\includegraphics[width=\unitlength,page=1]{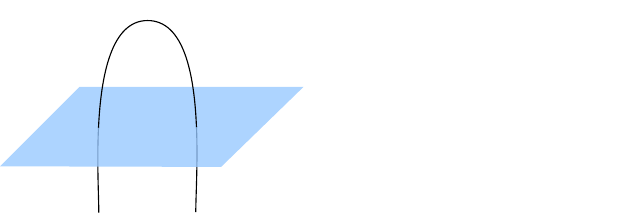}}%
    \put(0.1582459,0.30791231){\color[rgb]{0,0,0}\makebox(0,0)[lt]{\lineheight{1.25}\smash{\begin{tabular}[t]{l}$\alpha$\end{tabular}}}}%
    \put(0.03527171,0.01683846){\color[rgb]{0,0,1}\makebox(0,0)[lt]{\lineheight{1.25}\smash{\begin{tabular}[t]{l}$F'$\end{tabular}}}}%
    \put(0,0){\includegraphics[width=\unitlength,page=2]{tubing.pdf}}%
    \put(0.55481414,0.01683846){\color[rgb]{0,0,1}\makebox(0,0)[lt]{\lineheight{1.25}\smash{\begin{tabular}[t]{l}$F''$\end{tabular}}}}%
    \put(0,0){\includegraphics[width=\unitlength,page=3]{tubing.pdf}}%
    \put(0.48966719,0.17731388){\color[rgb]{0,0,0}\makebox(0,0)[lt]{\lineheight{1.25}\smash{\begin{tabular}[t]{l}tube\end{tabular}}}}%
  \end{picture}%
\endgroup%

\caption{Since the algebraic intersection number of $F'$ with each component $J_i$ of $J$ is zero, we can add a tube to $F'$ to find a surface $F''$ intersecting $J$ in two fewer points.}
	\label{fig:tubing}
\end{figure}

For each component $J_i$ of $J$, let $T_i$ be an embedded torus containing $J_i$ with surface framing equal to the zero framing, and such that the $T_i \cap F' = \emp$ and the tori $\{T_i\}$ are pairwise disjoint.  Then, we can tube the tori $\{T_i\}$ to $F'$, yielding a Seifert surface $F$ for $K$ such that the R-link $J$ is contained in $F$ such that each component has surface framing zero.  It follows that $J$ is a partial derivative for $K$, and thus by Proposition~\ref{meat}, $K$ has an R-link derivative.
\end{proof}

In Theorem 1.4 of~\cite{FHRkS}, Jeffrey Meier and the second author proved that if $K \cup J$ is a 2-component R-link and $K$ is fibered, then $K \cup J$ is stably equivalent to $K \cup L$, where $L$ is an R-link derivative contained in a fiber surface for $K$, from which it follows that the closed monodromy for $K$ extends over the handlebody determined by $L$.  Our work provides a shorter and far less technical proof of that fact:

\begin{theorem}\label{withJeff}
If $K \cup J$ is an R-link such that $K$ and $J$ are knots and $K$ is fibered, then $K \cup J$ is stably equivalent to $K \cup L$, where $L$ is an R-link derivative for $K$ contained in a fiber surface $F$.  In this case, the closed monodromy for $K$ extends over the handlebody determined by $L$.
\end{theorem}

\begin{proof}
Let $\wh F$ denote the closed fiber for the fibered 3-manifold $S^3_0(K)$.  As in~\cite{FHRkS} (which in turn adapts work in~\cite{GST}), we use a result of Scharlemann-Thompson~\cite{STLevel} to conclude that $J$ is isotopic into $\wh F$ with surface framing equal to the zero framing.  Thus, $J$ is a partial derivative for $K$, and by Proposition~\ref{meat}, there exists an R-link derivative $L \subset F$ such that $L$ is an R-link.

To see that the monodromy $\wh \varphi$ for $K$ extends as desired, let $g = |L|$ and note that $B_{K \cup L}$ has a relative handle decomposition with $(g+1)$ 2-handles, $(g+1)$ 3-handles, and a 4-handle.  To build $B_{K \cup L}$, start with relative 0-trace $B_0(K)$, the union of $I \X S^3$ and a 2-handle attached along $K$.  Next, let $H$ denote the abstract handlebody determined by $L$, and attach a copy of $H \X I$ to a collar neighborhood of $L \X I \subset \wh F \X I \subset S^3_0(K)$.  This attachment amounts to gluing $g$ 2-handles along $L$ followed by a single 3-handle.  Since the relative handle decomposition contains $g$ more 3-handles and a 4-handle, it follows that the resulting boundary component is $\#^g(S^1 \X S^2)$.  By the construction, this boundary is diffeomorphic to the union of $H \X \{0\}$, $H \X \{1\}$, and a product region $\wh F \X I$, since $S^3_0(K)$ is fibered.  Thus, we can cap off the boundary with another copy of $H \X I$ to obtain $B_{K \cup L}$.  In the process, we see that we have constructed $B_{K \cup L}$ by capping off each $\wh F$ fiber of $S^3_0(K)$ with a copy of $H$ determined by $L$, and we conclude that $\wh \varphi$ extends over handlebodies.
\end{proof}

\begin{remark}
The interested reader should note that embedded in the highly technical analysis contained in Section 4 of~\cite{FHRkS} are more fine-tuned statements about the closed monodromy $\wh \varphi$, including the assertion that the leveled curve $J$ is necessarily non-separating in $\wh F$ and restrictions on the iterated images of $J$ under $\wh \varphi$.  Our proof here does not recover these details, even though it suffices to give an alternate proof of the theorem.
\end{remark}

In the next two sections, we explore a more complicated version of the capping off in the proof of Theorem~\ref{withJeff} that occurs when we assume $S^3_0(K)$ is non-fibered, culminating in a Casson-Gordon-like theorem in the more general case.

\section{Singular fibrations, handlebody extensions, and generic maps}\label{sec:extend}

One goal of this paper is to extend Casson and Gordon's work to non-fibered knots.  For an arbitrary knot $K \subset S^3$, the exterior $S^3 \setminus K$ does not a priori admit a fibration; however, it does always admit a \emph{singular} fibration, on which we will focus in this section and the next.

For our purposes, a \emph{singular fibration} $p:Y \rightarrow S^1$ is a circle-valued Morse function such that fibers are connected.  In this case, regular fibers of $p$ are closed surfaces, and singular fibers are singular surfaces.  In addition, $p$ has the same number of index one and index two critical points (and no critical points of index zero or index three). If there exists a pair of regular fibers in $Y$ separating the set of index one critical points and the set of index two critical points, the singular fibration is called \emph{self-indexing}.  These two regular surfaces have minimal and maximal genera among regular surfaces, and thus they are called \emph{thin} and \emph{thick} surfaces, respectively.  A self-indexing singular fibration is an example of a circular Heegaard splitting (see~\cite{circle}).  It is well-known that any singular fibration can be homotoped to a self-indexing one.

Singular fibrations are used to define the \emph{Morse-Novikov} number of a knot: $MN(K)=\min\{$number of critical points in a singular fibration of $S^3\setminus K\}$. Clearly, a knot $K$ is fibered if and only if $MN(K)=0$; Goda~\cite{goda} has computed $MN(K)$ for all prime $K$ through ten crossings.

To prove the results in this section, we employ tools coming from \emph{generic maps} or \emph{Morse-2 functions} on 4-manifolds, a natural analogue of Morse functions.   A generic map (or Morse-2 function~\cite{GKM2}) is a smooth map $f:X \rightarrow \Sigma$, where $\Sigma$ is a compact surface, and $f$ is a submersion away from an embedded 1-manifold $Z_f \subset X$ (called the \emph{singular locus} of $f$).  In addition, for each point $y \in Z_f$, there are local coordinates $(t,x_1,x_2,x_3)$ about $y$ and $(u,v)$ about $f(y)$ such that $f(t,x_1,x_2,x_3) = (t,\pm x_1^2 \pm x_2^2 \pm x_3^2)$ or $f(t,x_1,x_2,x_3) = (t,x_1^3 + tx_1 \pm x_2^2 \pm x_3^2)$.  In the first case, $y$ is a called a \emph{fold point}, and in the second case, $y$ is called a \emph{cusp point}.  The singular locus is well-understood; it consists of fold circles and components containing fold arcs and isolated cusp points.  If $X$ has nonempty boundary, we require that $f$ map $\pd X$ into $\pd \Sigma$, where $f|_{\pd X}$ is a circle-valued Morse function.  See the comprehensive discussion in~\cite{BS} for further details.

Fold arcs are either \emph{definite} if all the signs agree or \emph{indefinite} if the signs differ; a fold arc does not have a well-defined index since it can be parametrized in two different ways.  An important observation is that in a neighborhood of a fold arc, $X$ is locally diffeomorphic to $Y \X I$ for a compression-body $Y$ (in the indefinite case) or a 3-ball $Y$ (in the definite case), and the function $f$ is given in local coordinates by $p \X \text{id}$, where $p:Y \rightarrow I$ is the natural Morse function with a single critical point.  Thus, for any 3-manifold $Y$ and Morse function $p:Y \rightarrow I$, there is a generic map $p \X \text{Id}:Y \X I \rightarrow I \X I$.  As in~\cite{BS}, we label images of fold arcs in the immersed 1-manifold $f(Z_f)$, called the \emph{base diagram}, with arrows, where the arrows point in the direction of 3-dimensional 2- or 3-handle attachment.

In what follows, we parametrize $S^1$ as $\R/2\pi \Z$, with $\theta \in \R/2\pi \Z$ corresponding to $e^{i\theta} \in S^1 \subset \C$.   In general, we assume that for a self-indexing singular fibration $p:Y \rightarrow S^1$, the thin surface is isotopic to $p^{-1}(0)$ and the thick surface is isotopic to $p^{-1}(\pi)$.  We say that a self-indexing singular fibration $p:Y \rightarrow S^1$ {\emph{extends over handlebodies to a generic map}} if there exists a compact 4-manifold $X$ with $\pd X = Y$ and a generic map $P:X \rightarrow S^1 \X I$ with the following properties:
\be
\item $P(\pd X) = S^1 \X \{0\}$ and $P|_{\pd X} = p$.
\item $P$ has a single definite fold circle whose image is $S^1 \X \{1\}$.
\item The directed segments $\{\theta\} \X I$ agree with the arrows of the base diagram $P(Z_P)$ (such a map is called \emph{directed}) and meet $P(Z_P)$ transversely except at singularities of $p$ and at cusp points.
\item The cusp points of $P(Z_P)$ are contained in $\{\pi\} \X I$, and the number cusp points equals the number of index one or index two critical points of $p$.
\ee

Condition (1) confirms that $P$ is an extension of $p$.  Condition (2) is necessary to ensure that the regular fibers of $P$ are connected surfaces.  Conditions (3) and (4) are necessary for a more complicated reason:  If $K$ is a knot in $S^3$ and $S^3_0(K)$ has a singular fibration $p$ that extends over handlebodies to $P:X \rightarrow S^1 \X I$, then there is a compact 4-manifold $B_P(K)$ obtained by gluing $X$ to the relative trace $B_0(K)$ along their common boundary component $S^3_0(K)$ via the identity map.  Conditions (3) and (4) guarantee that the 4-manifold $B_P(K)$ is a homotopy 4-ball.  Finally, it follows from condition (3) that away from cusp points and singularities of $p$, the inverse image $P^{-1}(\{\theta\} \X I)$ is a handlebody with boundary $p^{-1}(\theta)$, justifying the assertion that $P$ extends $p$ over handlebodies.  If $P_1$ is the projection of $P$ onto the $S^1$ coordinate, it also follows that $P_1:X \rightarrow S^1$ is a circular Morse function with both boundary critical points (coinciding with critical points of $p$) and interior critical points (coinciding with cusp points), and thus $P_1$ is an example of the type of extension defined in Section~\ref{sec:alternate}.  A good treatment of Morse theory on manifolds with boundary is contained in~\cite{BNR}.  If a singular fibration $p:S^3_0(K) \rightarrow S^1$ extends over handlebodies to $P:X \rightarrow S^1 \X I$, the base diagram $P(Z_P)$ must fit into the template shown in Figure~\ref{extension}.

\begin{figure}[h!]
	\centering
\begingroup%
  \makeatletter%
  \providecommand\color[2][]{%
    \errmessage{(Inkscape) Color is used for the text in Inkscape, but the package 'color.sty' is not loaded}%
    \renewcommand\color[2][]{}%
  }%
  \providecommand\transparent[1]{%
    \errmessage{(Inkscape) Transparency is used (non-zero) for the text in Inkscape, but the package 'transparent.sty' is not loaded}%
    \renewcommand\transparent[1]{}%
  }%
  \providecommand\rotatebox[2]{#2}%
  \newcommand*\fsize{\dimexpr\f@size pt\relax}%
  \newcommand*\lineheight[1]{\fontsize{\fsize}{#1\fsize}\selectfont}%
  \ifx\svgwidth\undefined%
    \setlength{\unitlength}{249.24756148bp}%
    \ifx\svgscale\undefined%
      \relax%
    \else%
      \setlength{\unitlength}{\unitlength * \real{\svgscale}}%
    \fi%
  \else%
    \setlength{\unitlength}{\svgwidth}%
  \fi%
  \global\let\svgwidth\undefined%
  \global\let\svgscale\undefined%
  \makeatother%
  \begin{picture}(1,1)%
    \lineheight{1}%
    \setlength\tabcolsep{0pt}%
    \put(0,0){\includegraphics[width=\unitlength,page=1]{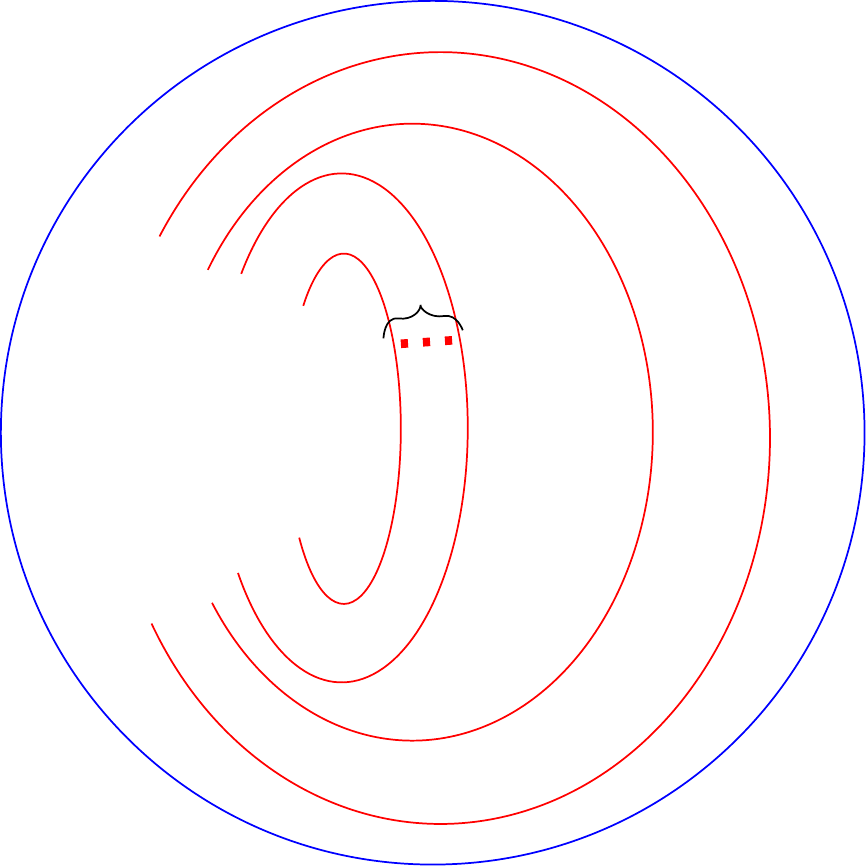}}%
    \put(0.4700389,0.65426289){\color[rgb]{0,0,0}\makebox(0,0)[lt]{\lineheight{1.25}\smash{\begin{tabular}[t]{l}$n$\end{tabular}}}}%
    \put(0,0){\includegraphics[width=\unitlength,page=2]{extension_new.pdf}}%
    \put(0.09947563,0.56784985){\color[rgb]{0,0,0}\rotatebox{-10.677957}{\makebox(0,0)[lt]{\lineheight{1.25}\smash{\begin{tabular}[t]{l}Many arcs,\end{tabular}}}}}%
    \put(0,0){\includegraphics[width=\unitlength,page=3]{extension_new.pdf}}%
    \put(0.08970365,0.38512571){\color[rgb]{0,0,0}\rotatebox{14.146177}{\makebox(0,0)[lt]{\lineheight{1.25}\smash{\begin{tabular}[t]{l}$n$ with cusps\end{tabular}}}}}%
    \put(0.2158647,0.29167981){\color[rgb]{0,0,0}\rotatebox{25.961961}{\makebox(0,0)[lt]{\lineheight{1.25}\smash{\begin{tabular}[t]{l}braid\end{tabular}}}}}%
    \put(0.19824885,0.67926967){\color[rgb]{0,0,0}\rotatebox{-26.124191}{\makebox(0,0)[lt]{\lineheight{1.25}\smash{\begin{tabular}[t]{l}braid\end{tabular}}}}}%
  \end{picture}%
\endgroup%

\caption{The base diagram of an extension $P:X \rightarrow S^1 \X I$}
	\label{extension}
\end{figure}

We also remark that in the case that $p$ is an honest fibration (so that $p$ has no critical points), then Casson-Gordon's Theorem~\ref{cassongordon} implies that the closed monodromy $\wh\varphi$ admits a handlebody extension $\Phi:H \rightarrow H$, and it is straightforward to construct a Morse-2 function $P:H \X_{\Phi} S^1 \rightarrow S^1 \X I$ with the desired properties such that $P_1$ is the natural projection to $S^1$, and the second coordinate function $P_2$ restricts to a standard Morse function on each handlebody fiber.  Thus, our definition coincides with Casson and Gordon's classification when $p$ is a fibration.

Next, we connect extensions over handlebodies to ideas about R-links discussed above.

\begin{lemma}\label{reverse_direction}
Suppose that $K \subset S^3$ admits a self-indexing singular fibration $p:S^3_0(K) \rightarrow S^1$ that extends over handlebodies to $P:X \rightarrow S^1 \X I$.  Then $K$ has an R-link derivative $L$ such that $B_L = B_P(K)$.
\end{lemma}

\begin{proof}
Let $\wh F \subset S^3_0(K)$ be the thin surface $p^{-1}(0)$ corresponding to the singular fibration $p$, with $k$ the genus of $\wh F$, and let $H = P^{-1}(\{0\} \X I)$ be the handlebody in $X$ whose boundary is $\wh F$.  The dual $K^* \in S^3_0(K)$ meets $\wh F$ in a single point, so that $F = \wh F \setminus K^*$ is a Seifert surface for $K$ in $S^3$.  Choose a link $L \subset  \wh F \setminus K^*$ such that $L$ is a $k$-component link framed by $F$ and bounding a collection of disks in $H$ cutting $H$ into a 3-ball.  We claim that $L$ is an R-link derivative for $K$.

We use a collar neighborhood $I \X H \subset X$ to begin building a relative handle decomposition for $B_P(K)$.  Observe that $X$ contains the union of $S^3_0(K) \X I$ along with $k$ 2-handles attached to $L$ with the surface framing of $L$ in $F$.  In addition, the cores of these 2-handles cut $I \X H$ into $I \X B^3$ which can be capped over with a 3-handle.  Let $X_0 \subset X$ be the union of $S^3_0(K) \X I$, these $k$ 2-handles, and the 3-handle.  Then $B_0(K) \cup X_0$ can be built with $(k+1)$ 2-handles attached along $K \cup L \subset S^3$ and a 3-handle.  Thus, to prove that $K \cup L$ (and thus $L$) is an R-link, it suffices to show that $X_1 = X \setminus X_0$ is a 1-handlebody built from a 0-handle and $k$ 1-handles.

For small $\delta,\eps > 0$, we may assume that $S^3_0(K) \X I$ is given by $P^{-1}(S^1 \X [0,\delta])$, and the collar neighborhood $I \X H \subset X$ is given by $P^{-1}([-\eps,\eps] \X [\delta,1])$, so that $X_0 = P^{-1}(S^1 \X [0,\delta]) \cup P^{-1}([-\eps,\eps] \X [\delta,1])$.  It follows that $X_1 = P^{-1}([\eps,2\pi-\eps] \X [\delta,1])$.  Decompose $X_1$ into
\begin{eqnarray*}
X_- &=& P^{-1}([\eps,\pi-\eps] \X [\delta,1]) \\
X_{\pi} &=& P^{-1}([\pi-\eps,\pi+\eps] \X [\delta,1]) \\
X_+ &=& P^{-1}([\pi+\eps,2\pi-\eps] \X [\delta,1]).
\end{eqnarray*}

Suppose $g$ is the genus of the thick surface of $p$, and let $H_{\pm}$ be the genus $g$ handlebodies given by $P^{-1}(\{\pi \pm \eps\} \X I)$.  Then, since $P$ is a directed generic map (whose base diagram is divided as shown in Figure~\ref{sections}), it follows that $X_{\pm}$ is diffeomorphic to $H_{\pm} \X I$, so that $X_1$ deformation retracts to $X_{\pi}$.  Using the base diagram for $X_{\pi}$, a small neighborhood of the definite fold contributes a 0-handle to $X_{\pi}$, a small neighborhood of an indefinite fold without a cusp contributes a 1-handle to $X_{\pi}$, and a small neighborhood of an indefinite fold with a cusp does not change the topology of $X_{\pi}$.  Thus, $X_{\pi}$ is a 1-handlebody built from a 0-handle and $k$ 1-handles, where $k$ is the number of indefinite folds without cusps, completing the proof.
\end{proof}

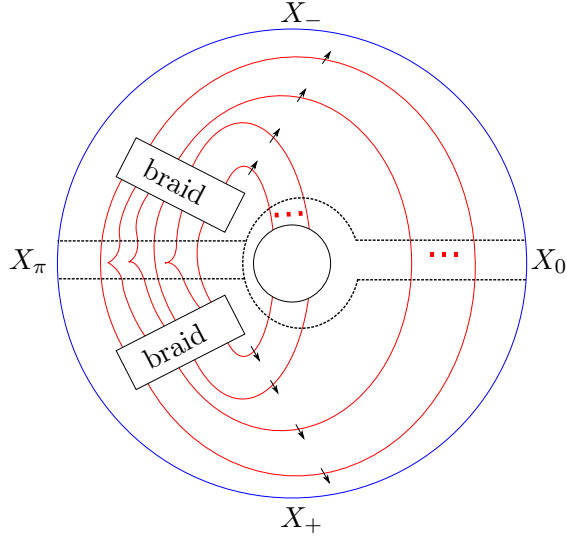
\begin{figure}[h!]
	\centering
\begingroup%
  \makeatletter%
  \providecommand\color[2][]{%
    \errmessage{(Inkscape) Color is used for the text in Inkscape, but the package 'color.sty' is not loaded}%
    \renewcommand\color[2][]{}%
  }%
  \providecommand\transparent[1]{%
    \errmessage{(Inkscape) Transparency is used (non-zero) for the text in Inkscape, but the package 'transparent.sty' is not loaded}%
    \renewcommand\transparent[1]{}%
  }%
  \providecommand\rotatebox[2]{#2}%
  \newcommand*\fsize{\dimexpr\f@size pt\relax}%
  \newcommand*\lineheight[1]{\fontsize{\fsize}{#1\fsize}\selectfont}%
  \ifx\svgwidth\undefined%
    \setlength{\unitlength}{226.5575499bp}%
    \ifx\svgscale\undefined%
      \relax%
    \else%
      \setlength{\unitlength}{\unitlength * \real{\svgscale}}%
    \fi%
  \else%
    \setlength{\unitlength}{\svgwidth}%
  \fi%
  \global\let\svgwidth\undefined%
  \global\let\svgscale\undefined%
  \makeatother%
  \begin{picture}(1,0.88539938)%
    \lineheight{1}%
    \setlength\tabcolsep{0pt}%
    \put(0,0){\includegraphics[width=\unitlength,page=1]{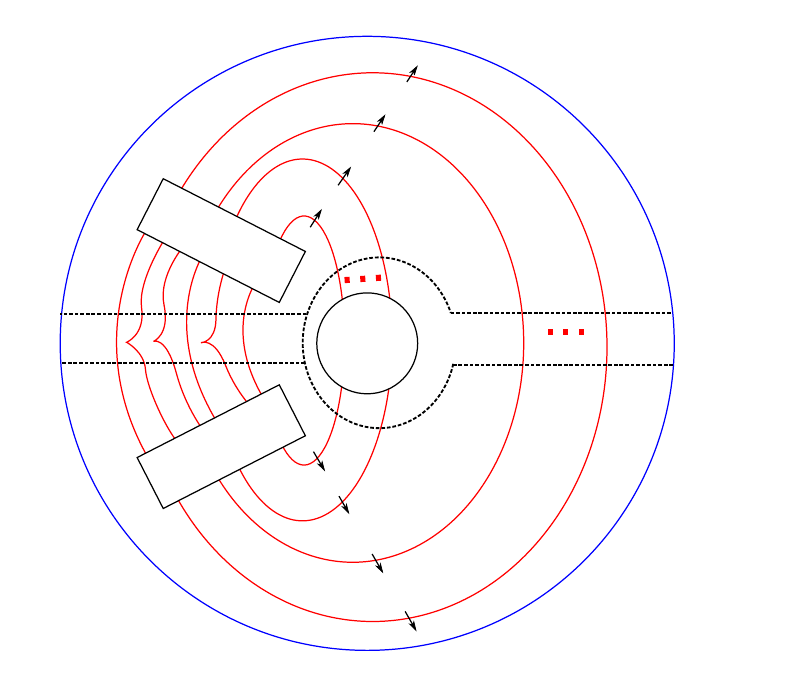}}%
    \put(0.86388334,0.43977726){\color[rgb]{0,0,0}\makebox(0,0)[lt]{\lineheight{1.25}\smash{\begin{tabular}[t]{l}$X_0$\end{tabular}}}}%
    \put(-0.00375099,0.43977726){\color[rgb]{0,0,0}\makebox(0,0)[lt]{\lineheight{1.25}\smash{\begin{tabular}[t]{l}$X_\pi$\end{tabular}}}}%
    \put(0.44786461,0.85106683){\color[rgb]{0,0,0}\makebox(0,0)[lt]{\lineheight{1.25}\smash{\begin{tabular}[t]{l}$X_-$\end{tabular}}}}%
    \put(0.44786461,0.01065721){\color[rgb]{0,0,0}\makebox(0,0)[lt]{\lineheight{1.25}\smash{\begin{tabular}[t]{l}$X_+$\end{tabular}}}}%
    \put(0.21555816,0.59587067){\color[rgb]{0,0,0}\rotatebox{-27.531285}{\makebox(0,0)[lt]{\lineheight{1.25}\smash{\begin{tabular}[t]{l}braid\end{tabular}}}}}%
    \put(0.23002948,0.27115881){\color[rgb]{0,0,0}\rotatebox{27.367076}{\makebox(0,0)[lt]{\lineheight{1.25}\smash{\begin{tabular}[t]{l}braid\end{tabular}}}}}%
  \end{picture}%
\endgroup%

\caption{The base diagram of $P$ divided into regions lifting to $X_0$, $X_{\pi}$, and $X_{\pm}$.}
	\label{sections}
\end{figure}

This proposition establishes the first half of the equivalence offered in Theorem~\ref{fibration1} below.  To prove the other direction, we require the power afforded to us by Waldhausen's Theorem.  See~\cite{schleimer} for further details about the theorem.

\begin{theorem}\label{wald}\cite{waldhausen}
For every $g,k$ with $g \geq k \geq 0$, the 3-manifold $\#^k(S^1 \X S^2)$ has a unique genus $g$ Heegaard splitting up to isotopy.
\end{theorem}

We will also need to build a handlebody extension $P$ of a given singular fibration, and we accomplish this task in pieces.  The most important component of this piecewise construction is given in Lemma~\ref{capoff} below.  For our purposes, a Morse function $h:Y \rightarrow [0,3]$ is \emph{self-indexing} if $h$ has one index zero critical point occurring at $h^{-1}(0)$, one index three critical point occurring at $h^{-1}(3)$, and if $h^{-1}(t)$ is a regular surface separating the index one critical points from the index two critical points for any $t \in [1,2]$.  Note that this definition is somewhat nonstandard.

\begin{lemma}\label{capoff}
Let $g \geq k \geq 0$ and let $h:\#^k(S^1\X S^2) \rightarrow [0,3]$ be a self-indexing Morse function.  Viewing $h$ as a map to $(I \X \{0\}) \cup (\{1\} \X I) \cup (I \X \{1\}) \subset I \X I$, with the natural time direction reversed in $I \X \{1\}$, there exists a generic map $Q:\natural^k(S^1 \X B^3) \rightarrow I \X I$ extending $h$, where the base diagram of $Q$ has $k$ indefinite folds without cusps, $g-k$ indefinite folds containing a single cusp, and a single definite fold mapping to $\{0\} \X I$, as shown in Figure~\ref{handlebody}.
\end{lemma}

\begin{proof}
Choose a Heegaard diagram $(\A,\n)$ compatible with $h$.  By Waldhausen's Theorem, $(\A,\n)$ is handleslide equivalent to a standard diagram $(\A',\n')$, in which $\A'_i = \n'_i$ for $1 \leq i \leq k$ and $\A'_i \cap \n'_j = \delta_{ij}$ for $k+1 \leq i,j \leq g$.  Let $H = h^{-1}([0,1])$, let $H' = h^{-1}([2,3])$, and let $\Sigma = h^{-1}(3/2)$, so that $h^{-1}([1,2]) = \Sigma \X I$ and $\#^k(S^1 \X S^2)$ decomposes as $H \cup (\Sigma \X I) \cup H'$.  Since $\A$ and $\A'$ are equivalent, there exists a 1-parameter family of Morse functions $h_t: H \rightarrow I$, where $t \in [0,1/3]$, $h_0 = h|_H$, and the cut system $\A'$ is compatible with $h_{1/3}$.  Then $h_t$ gives rise to $Q:H \X [0,1/3] \rightarrow I \X [0,1/3]$ by $Q(x,t) = (h_t(x),t)$, where the base diagram for $Q$ is braided as shown in the bottom third of Figure~\ref{handlebody}.

A parallel construction produces a 1-parameter family of Morse functions $g_t:H' \rightarrow I$, where $t \in [2/3,1]$, $g_1$ is equal to $h|_{H'}$ with time direction reversed, and $g_{2/3}$ is compatible with the cut system $\beta'$.  As such, $g_t$ gives rise to $Q:H' \X [2/3,1] \rightarrow I \X [2/3,1]$ by $Q(x,t)= (g_t(x),t)$, where the base diagram for this portion of $Q$ is braided as shown in the top third of Figure~\ref{handlebody}.

The last step is to complete the extension of $Q$ over $\natural^k(S^1 \X B^3)$.  Along the boundary of our current extension $Q$, we have $(H \X \{1/3\}) \cup (\Sigma \X [1/3,2/3]) \cup (H' \X \{2/3\})$ mapping to $(I \X \{1/3\}) \cup (\{1\} \X [1/3,2/3]) \cup (I \X \{2/3\})$, where the standard Heegaard diagram $(\A',\n')$ is compatible with the induced Morse function along the boundary.  For each of the $g-k$ pairs of curves $\A'_i \cap \n'_j = \delta_{ij}$, where $k+1 \leq i,j \leq g$, we can extend $Q$ by introducing a fold arc with a single cusp, which has no effect on the topology of the domain but modifies the range to be $(I \X [0,1/3]) \cup ([1/2,1] \X [1/3,2/3]) \cup (I \X [2/3,1])$.  Finally, for each of the $k$ curves in $\A' \cap \n'$, we can extend $Q$ by mapping a 4-dimensional 3-handle to a rectangular region containing a single vertical indefinite fold arc, followed by a 4-handle mapping to a vertical definite fold arc in $\{0\} \X [1/3,2/3]$.  The result is a map $Q:X \rightarrow I \X I$, where $X$ is obtained by attaching $k$ 3-handles and a 4-handle to a collar of $\#^k(S^1 \X S^2)$, so that $X = \natural^k(S^1 \X B^3)$, as desired. 
\end{proof}

\begin{figure}[h!]
	\centering
\begingroup%
  \makeatletter%
  \providecommand\color[2][]{%
    \errmessage{(Inkscape) Color is used for the text in Inkscape, but the package 'color.sty' is not loaded}%
    \renewcommand\color[2][]{}%
  }%
  \providecommand\transparent[1]{%
    \errmessage{(Inkscape) Transparency is used (non-zero) for the text in Inkscape, but the package 'transparent.sty' is not loaded}%
    \renewcommand\transparent[1]{}%
  }%
  \providecommand\rotatebox[2]{#2}%
  \newcommand*\fsize{\dimexpr\f@size pt\relax}%
  \newcommand*\lineheight[1]{\fontsize{\fsize}{#1\fsize}\selectfont}%
  \ifx\svgwidth\undefined%
    \setlength{\unitlength}{118.00559937bp}%
    \ifx\svgscale\undefined%
      \relax%
    \else%
      \setlength{\unitlength}{\unitlength * \real{\svgscale}}%
    \fi%
  \else%
    \setlength{\unitlength}{\svgwidth}%
  \fi%
  \global\let\svgwidth\undefined%
  \global\let\svgscale\undefined%
  \makeatother%
  \begin{picture}(1,0.96431281)%
    \lineheight{1}%
    \setlength\tabcolsep{0pt}%
    \put(0,0){\includegraphics[width=\unitlength,page=1]{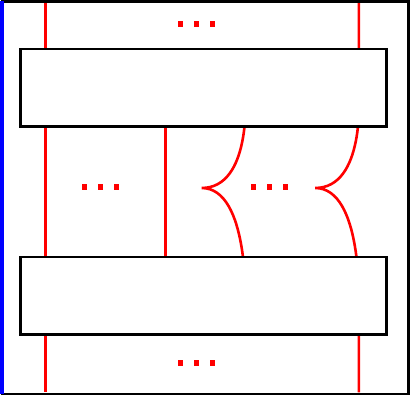}}%
    \put(0.21347971,0.54740081){\color[rgb]{1,0,0}\makebox(0,0)[lt]{\lineheight{1.25}\smash{\begin{tabular}[t]{l}$k$\end{tabular}}}}%
    \put(0.595479,0.54740081){\color[rgb]{1,0,0}\makebox(0,0)[lt]{\lineheight{1.25}\smash{\begin{tabular}[t]{l}$g-k$\end{tabular}}}}%
        \put(0.39,0.715){\makebox(0,0)[lt]{\lineheight{1.25}\smash{\begin{tabular}[t]{l}braid\end{tabular}}}}%
        \put(0.39,0.215){\makebox(0,0)[lt]{\lineheight{1.25}\smash{\begin{tabular}[t]{l}braid\end{tabular}}}}%
    \put(0,0){\includegraphics[width=\unitlength,page=2]{handlebody.pdf}}%
  \end{picture}%
\endgroup%

\caption{A generic map from a 4-dimensional 1-handlebody to $I \X I$.}
	\label{handlebody}
\end{figure}

Now, we put all of these tools together to prove the next theorem.  Together with Theorem~\ref{hr}, this establishes one of the two versions of Theorem~\ref{extensions} described in the introduction.

\begin{theorem}\label{fibration1}
Suppose $K$ is a knot in $S^3$.  Then $K$ has an R-link derivative $L$ if and only if there exists a self-indexing singular fibration $p$ of $S^3_0(K)$ that extends over handlebodies to a map $P:X \rightarrow S^1 \X I$.  Moreover, the 4-manifolds $B_L$ and $B_P(K)$ are diffeomorphic.
\end{theorem}

\begin{proof}
The reverse implication is the content of Lemma~\ref{reverse_direction}.  For the other implication, suppose that $K$ has an R-link derivative $L$ contained in a Seifert surface $F$, and let $\wh F$ denote the closed surface in $S^3_0(K)$ obtained by capping $\pd F$ with a disk in the surgery solid torus.  Then $S^3_0(K) \setminus \wh F$ is connected, and there is a Morse function  $p_1:S^3_0(K) \setminus \wh F \rightarrow I$ with index one and two critical points and such that the two parallel copies of $\wh F$ composing $\pd(S^3_0(K) \setminus \wh F)$ are $p_1^{-1}(\{0,1\})$.  Generically, there is a homotopy from $p_1$ to a map $p_2$ in which all index one critical points occur below the index two critical points, and finally re-gluing $S^3_0(K) \setminus \wh F$ along $\wh F$ yields a self-indexing singular fibration $p:S^3_0(K) \rightarrow S^1$ in which $\wh F$ is a thin surface.  Note that $H_1(S^3_0(K)) = \Z$, which implies that the fibers of $p$ are connected.

The R-link $L$ is stably equivalent to the R-link $K \cup L$ by Lemma~\ref{kremove}, where $B_{K \cup L}$ has a relative handle decomposition with $k+1$ 2-handles, $k+1$ 3-handles, and a 4-handle, with $k = |L|$.  Flipping this decomposition yields a handle decomposition for $B_{K \cup L}$ with a 0-handle, $k+1$ 1-handles, and $k+1$ 2-handles.  Consider the 4-manifold $X \subset B_{K \cup L}$ consisting of the 0-handle, $k+1$ 1-handles, and $k$ 2-handles, so that $B_{K \cup L}$ is obtained from $X$ by attaching a 2-handle along the dual $K^*$.  It follows that $\pd X = S^3_0(K)$.  We claim that $p$ extends over handlebodies to a map $P:X \rightarrow S^1 \X I$.

To build $P$, we note first that there is a natural map $P$ from a collar $S^3_0(K) \X I \subset X$ to an annulus by taking $p \X \text{Id}$.  The 4-manifold $X$ has a relative handle decomposition obtained by attaching $k$ 2-handles to $L \X \{1\} \subset \wh F \X \{1\} \subset S^3_0(K) \X I$, followed by $k+1$ 3-handles and a 4-handle.  A collar neighborhood $I \X \wh F$ in $S^3_0(K) \X \{1\}$ bounds a collar neighborhood $I \X H$ of a 3-dimensional genus $k$ handlebody $H$ in $X$, where $I \X H$ can be constructed by attaching the $k$ 2-handles to $L \subset \wh F$ and capping off the resulting $I \X B^3$ with a 3-handle (as in the proof of Lemma~\ref{reverse_direction}).  Let $X_0 = (S^3_0(K) \X I) \cup (I \X H)$.  There is a natural generic map from $I \X H$ to $I \X I$ with $k$ indefinite folds and one definite fold, which can be used to extend $P$ to a generic map from $X_0$ to the region $A_0$, which is shaped like a magnifying glass and shown at left in Figure~\ref{magnify}.

\begin{figure}[h!]
	\centering
	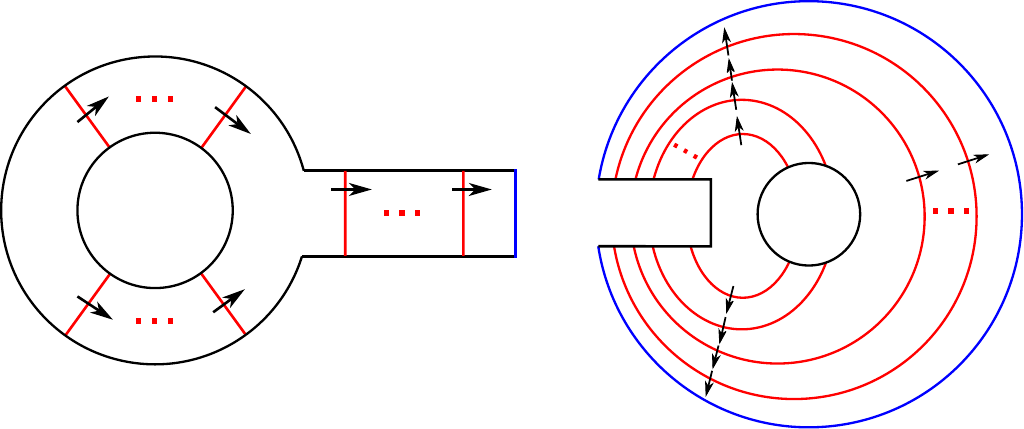
\caption{At left, the region $A_0$, and at right, its image $A_1$ under diffeomorphism.}
\label{magnify}
\end{figure}

Now, there is a diffeomorphism from $A_0$ to the region $A_1$, transforming the base diagram as shown in the right of Figure~\ref{magnify}.  In an abuse of notation, let $P$ be the generic map from $X_0$ to $A_1$.  Finally, let $X_1 = X \setminus X_0$.  Then $X_1$ is a 4-dimensional genus $k$ handlebody, so that $\pd X_0 = \#^k(S^1 \X S^2)$.  Since the original singular fibration $p$ was self-indexing, it follows that our constructed map $P$ induces a self-indexing Morse function on $\pd X_0$, which can be capped off by a generic map $Q$ of the form given in Lemma~\ref{capoff} and shown in Figure~\ref{handlebody}, extending $P$ over $X$.  By inspection, $p$ extends over handlebodies to $P:X \rightarrow S^1 \X I$, and by construction we have that $B_{L} = B_{K \cup L} = B_P(K)$, completing the proof.
\end{proof}

\section{An alternate perspective: circular Morse functions}\label{sec:alternate}

In this section, we provide another interpretation of Theorem~\ref{extensions}, proving Theorem~\ref{fibration2} following the techniques in~\cite{maggie}.  This theorem asserts that $K$ is handle-ribbon in a homotopy 4-ball if and only if there exists a singular fibration of $S^3_0(K)$ that extends over handlebodies via a circular Morse function (defined below).  Because the forward direction of Theorem~\ref{fibration2} also follows from Theorems~\ref{hr} and~\ref{fibration1}, we omit some details in this argument. This section is meant to illustrate a different perspective; both proofs are constructions, but the proof here builds a 1-parameter family of singular fibrations instead of a generic map to an annulus.  For a given handle-ribbon knot $K$ and singular fibration $p:S^3_0(K) \rightarrow S^1$, the 1-parameter family constructed here could be used to construct a generic map, possibly with smaller genus fibers than an extension $P:X \rightarrow S^1 \X I$ of $p$ constructed in Section~\ref{sec:extend}, but at the expense of having a more complicated graphic (see Remark~\ref{thinremark} below).

Suppose that $Y$ is a 3-manifold and $p_t:Y \rightarrow S^1$ is a smooth family of singular fibrations such that each $p_t$ has the same set of fibers.  Suppose further that $p_{t_0}$ has an index one or two critical point, so that there is a singular fiber $p_{t_0}^{-1}(\theta)$ containing a cone point $q$.  We say that this singularity is \emph{type II} or \emph{compressing} (in the language of~\cite{maggie}) if a neighborhood of $q$ in $\bigcup_t p_t^{-1}(\theta)$ intersects $p_t^{-1}(\theta)$ in an annulus for $t\in (t_0-\eps,t_0)$ and a pair of disks if $t \in (t_0,t_0+\eps)$. See Figure~\ref{fig:typeiicone}. The type of a cone is a stable property in that if the type of a cone ever changes as $t$ increases, then there is some $t',\theta'$ for which the type of a cone in $p_{t'}^{-1}(\theta')$ is undefined.  In this case, $\bigcup_t p_{t}^{-1}(\theta')$ is a singular 3-manifold.  A type II cone $p_{t}^{-1}(\theta)$ contributes a 3-dimensional 2-handle to $\bigcup_tp_t^{-1}(\theta)$.  See~\cite{maggie} for further details.

\begin{figure}[h!]
	\centering
\begingroup%
  \makeatletter%
  \providecommand\color[2][]{%
    \errmessage{(Inkscape) Color is used for the text in Inkscape, but the package 'color.sty' is not loaded}%
    \renewcommand\color[2][]{}%
  }%
  \providecommand\transparent[1]{%
    \errmessage{(Inkscape) Transparency is used (non-zero) for the text in Inkscape, but the package 'transparent.sty' is not loaded}%
    \renewcommand\transparent[1]{}%
  }%
  \providecommand\rotatebox[2]{#2}%
  \newcommand*\fsize{\dimexpr\f@size pt\relax}%
  \newcommand*\lineheight[1]{\fontsize{\fsize}{#1\fsize}\selectfont}%
  \ifx\svgwidth\undefined%
    \setlength{\unitlength}{247.27627251bp}%
    \ifx\svgscale\undefined%
      \relax%
    \else%
      \setlength{\unitlength}{\unitlength * \real{\svgscale}}%
    \fi%
  \else%
    \setlength{\unitlength}{\svgwidth}%
  \fi%
  \global\let\svgwidth\undefined%
  \global\let\svgscale\undefined%
  \makeatother%
  \begin{picture}(1,0.39539657)%
    \lineheight{1}%
    \setlength\tabcolsep{0pt}%
    \put(0,0){\includegraphics[width=\unitlength,page=1]{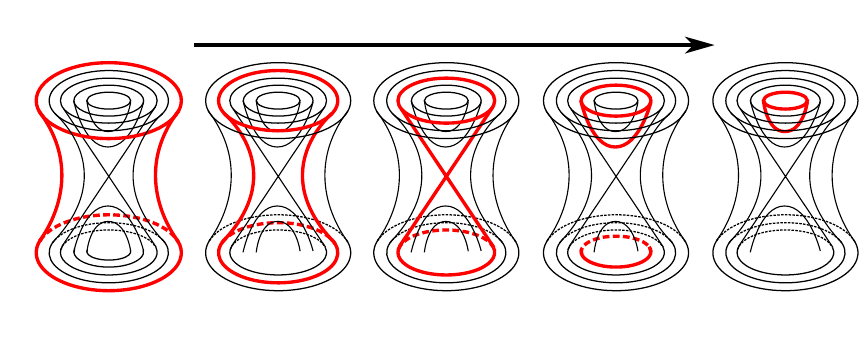}}%
    \put(0.51241147,0.35598649){\color[rgb]{0,0,0}\makebox(0,0)[lt]{\lineheight{1.25}\smash{\begin{tabular}[t]{l}$t$\end{tabular}}}}%
    \put(-0.00430573,0.3423763){\color[rgb]{1,0,0}\makebox(0,0)[lt]{\lineheight{1.25}\smash{\begin{tabular}[t]{l}$f_t^{-1}(\theta)$\end{tabular}}}}%
    \put(0,0){\includegraphics[width=\unitlength,page=2]{typeiicone.pdf}}%
    \put(0.47195547,0.01028579){\color[rgb]{0,0,0}\makebox(0,0)[lt]{\lineheight{1.25}\smash{\begin{tabular}[t]{l}$t=t_0$\end{tabular}}}}%
  \end{picture}%
\endgroup%

\caption{A type II critical point of $p_{t_0}$. We draw a neighborhood of the critical point for $t\in[t_0-\eps,t_0+\eps]$.}
	\label{fig:typeiicone}
\end{figure}

Suppose now $X$ is a compact 4-manifold and $\Ph:X \rightarrow S^1$ is a circular Morse function, with $Y = \pd X$ and $p = \Ph\,|_Y$.  On a collar neighborhood $Y \X I \subset X$, of $Y$, where $Y = Y \X \{0\}$, $\Ph$ induces a smooth family of singular fibrations $p_t$, where $p = p_0$ and each $p_t$ has the same fibers as $p_0$, up to reparametrization of $S^1$.  For each boundary critical point $q$ of $\Ph$, we assume by genericity that each singularity of $p_t$ has a well-defined type for small $t > 0$, and we say the boundary critical point is type II if the corresponding singularity of $p_t$ has type II for small $t > 0$. If $y_1\in p_0^{-1}(\theta)$ is an index one critical point of $p_0$ which is type II, then $y_1$ contributes a 1-handle to $\cup_tp_t^{-1}(\theta-\eps,\eta+\eps)$. If $y_1\in p_0^{-1}(\theta)$ is an index two critical point of $p_0$ which is type II, then $y_2$ contributes no handles to $\cup_tp_t^{-1}(\theta-\eps,\eta+\eps)$. In the language of~\cite{BNR}, index one type II boundary critical points contribute right half-handles while index two type II boundary critical points contribute left half-handles.

Let $K\subset S^3$ be a knot and suppose there exists a singular fibration $p$ of $S^3_0(K)$ that extends to a circular Morse function $\Ph:X\to S^1$ so that every critical point of $\Ph$ is either a boundary critical point of type II or an interior critical point of index two, the number of boundary critical points is twice the number of interior critical points, and the regular fibers of $\Ph$ are 3-dimensional handlebodies. Then we say that {\emph{$p$ extends over handlebodies to the circular Morse function $\Ph$}}.

\begin{remark}
If $\Ph:X\to S^1$ is a circular Morse function with $n$ index one type II boundary critical points, $n$ index two type II boundary critical points, $m$ interior index two critical points, and no other critical points, then $\chi(X)=m-n$. Therefore, requiring that the number of boundary critical points be twice the number of interior critical points (i.e. $n=m$) is equivalent to requiring $\chi(X)=0$.
\end{remark}

\begin{remark}
Despite the fact that the circular Morse function $\Ph:X \rightarrow S^1$ could also be described as a singular fibration, we use \emph{singular fibration} exclusively for a circular Morse function $p:Y \rightarrow S^1$, where $Y$ is a 3-manifold, as an attempt to avoid confusion.  Note that the definition here is more relaxed than extending over handlebodies to an annular generic map as in Section~\ref{sec:extend}; if a self-indexing singular fibration $p:Y \rightarrow S^1$ extends over handlebodies to the generic map $P:X \rightarrow S^1 \X I$, then the projection $P_1:X \rightarrow S^1$ is an extension of $p$ over handlebodies to the circular Morse function $P_1$.  On the other hand, a singular fibration $p$ need not be self-indexing to admit an extension over handlebodies to a circular Morse function, for example.
\end{remark}

We can now state the main theorem from this section.

\begin{theorem}\label{fibration2}
A knot $K \subset S^3$ is handle-ribbon in a homotopy 4-ball if and only if there exists a singular fibration $p:S^3_0(K) \rightarrow S^1$ that admits an extension over handlebodies to a circular Morse function $\Ph:X \rightarrow S^1$.
\end{theorem}

We prove the reverse direction of Theorem~\ref{fibration2} in the following lemma.

\begin{lemma}\label{onedirection}
Let $K\subset S^3$ be a knot and $p$ a singular fibration of $S^3_0(K)$. Suppose $p$ extends over handlebodies to a circular Morse function $\Ph:X\to S^1$. Then $K$ is handle-ribbon in a homotopy 4-ball.
\end{lemma}

\begin{proof}
It suffices to show that $X$ has a handle decomposition with a single 0-handle, $c+1$ 1-handles, and $c$ 2-handles for some integer $c$.  In this case, gluing the relative trace $B_0(K)$ to $X$ along $S^3_0(K)$ yields a compact 4-manifold $B$ built from a single 0-handle, $c+1$ 1-handles, and $c+1$ 2-handles with $\pd B = S^3$, so that $B$ is a homotopy 4-ball.  Moreover, $K$ bounds a cocore of a 2-handle in $B$, implying that $K$ is handle-ribbon in $B$.

To see that $X$ has such a handle decomposition, let $H$ be a regular fiber of $X$, so that $X$ contains a collar neighborhood $H \X [-\eps,\eps]$.  Since $H$ is a handlebody by assumption, $H \X [0,\eps]$ has a handle decomposition with one 0-handle and $g$ 1-handles, where $g$ is the genus of $H$.  Suppose that $\Ph:X \rightarrow S^1$ has $n$ interior index two critical points and $2n$ boundary critical points; as critical points of $p$, the boundary critical points consist of $n$ index one critical points and $n$ index two critical points.

Since the boundary critical points are type II, and using the language and machinery from~\cite{BNR}, passing through each index one boundary critical point corresponds to attaching a right half-handle of index one, whereas passing through each index two boundary critical point corresponds to attaching a left half-handle of index three.  By Lemmas 2-18 and 2-19 of~\cite{BNR}, a right 1-half-handle attachment corresponds to attaching a 4-dimensional 1-handle, while a left half-handle attachment has no effect on topology.  Thus, let $X' = X \setminus (H \X  [-\eps,0])$.  Then $X'$ is obtained from $H \X [0,\eps]$ by attaching $n$ 1-handles (corresponding to boundary critical points of index one) and $n$ 2-handles (corresponding to the interior critical points of $\Ph$).

Suppose the 3-dimensional handlebody $H$ is a regular neighborhood of the graph $\Gamma$, where $\Gamma$ has one 0-cell and $g$ 1-cells, $e_1,\dots,e_g$.  Since $X = X' \cup (H \X [-\eps,0])$, we can obtain $X$ from $X'$ by attaching a 1-handle along 3-dimensional neighborhoods of the vertex of $\Gamma$ in $H \X \{-\eps\}$ and $H \X \{0\}$, followed by attaching $g$ 2-handles along the curves $e_i \X [-\eps,0]$.  We conclude that $X$ has a handle decomposition with a single 0-handle, $(g+n+1)$ 1-handles, and $(g+n)$ 2-handles, completing the proof.
\end{proof}

Next, we prove the longer and more complicated direction of Theorem~\ref{fibration2}.

\begin{proposition}\label{otherdirection}
Suppose that $K$ is handle-ribbon in a homotopy 4-ball $B$.  Then there exists a singular fibration $p:S^3_0(K) \rightarrow S^1$ and an extension $\Ph:X \rightarrow S^1$ of $p$ over handlebodies.
\end{proposition}

\begin{proof}
Suppose that $K$ is handle-ribbon in a homotopy 4-ball $B$.  As in Lemma~\ref{rhandle}, we have that $X=B\setminus D$ admits a relative handle decomposition with $c$ 2-handles, $c+1$ 3-handles, and a single 4-handle (for some $c$). This handle decomposition induces a self-indexing Morse function $h$ on $X$, where $h:X \rightarrow [1,4]$, $\pd X = h^{-1}(1)$, $h$ has $c$ index two critical points at height 2 corresponding to the 2-handles, $c+1$ index three critical points at height 3 corresponding to the 3-handles, and one index four critical point at height 4 corresponding to the 4-handle.  For notational convenience, we let $Y_t = h^{-1}(t)$, noting that $Y_t$ is a smooth 3-manifold for $t \neq 2,3,4$ and $Y_1 = \pd X$.  In this case, $Y_t \cong \pd X$ for $t \in [1,2)$, $Y_t \cong \#^{c+1}(S^1 \X S^2)$ for $t\in (2,3)$, and $Y_t \cong S^3$ for $t\in (3,4)$.

We will construct a family of singular fibrations $p_t:Y_t \rightarrow S_1$, with $t \in [1,4]$, so that $\Ph:X \rightarrow S^1$ given by $\Ph(x) = p_{h(x)}(x)$ extends $p_1$ over handlebodies as in the definition above.  While we assume the singular fibration $p_1$ of $\pd X$ has connected fibers, we relax the restriction that a singular fibration must have connected fibers for other $p_t$ in this family.

\textbf{Step 1: Constructing the initial singular fibration $p_1$.}  View the attaching circles of the 2-handles as a framed $c$-component link $L$ in $S^3_0(K) = \partial X$.  As in the proof of Theorem~\ref{hr}, there exists a Seifert surface $F$ for $K$ such that the capped off surface $\wh F \subset S^3_0(K)$ contains $L$, with the surface framing agreeing with the framing of $L$.  Since $[\wh F]$ is a generator of $H_2(S^3_0(K)) = \Z$, there exists singular fibration $p_1:S^3_0(K) \rightarrow S^1$ such that $\wh F$ is a regular level.  Suppose $p_1$ has $2n$ singularities (of which $n$ are index one and $n$ are index two).

\textbf{Step 2: Extending $p_1$ over the 2-handles of $X$.}  To this end, for some small $\eps > 0$, we let $p_t$ have the same level sets as $p_1$ for all $t \in [1,2-\eps]$. For $t\in (1,2-\eps]$, we reparametrize $p_t$ near the singularities of $p_1$ so that each singularity of $p_t$ is type II (as described above).  As $t$ increases from $2-\eps$ to $2+\eps$, we introduce $c$ singularities of type II into $p_t$, one corresponding to each 2-handle in $X$, as illustrated in Figure~\ref{compress}, which depicts a $p_t$ in a neighborhood of each index two critical point of $h$, where $t \in (2-\eps,2+\eps)$.  Recall that a neighborhood of a type II cone point evolves from an annulus to two disks as $t$ increases, and in this case, a core of the annulus is a component of the attaching link $L$, whose framing agrees with the surface framing.

\begin{figure}[h!]
	\centering
	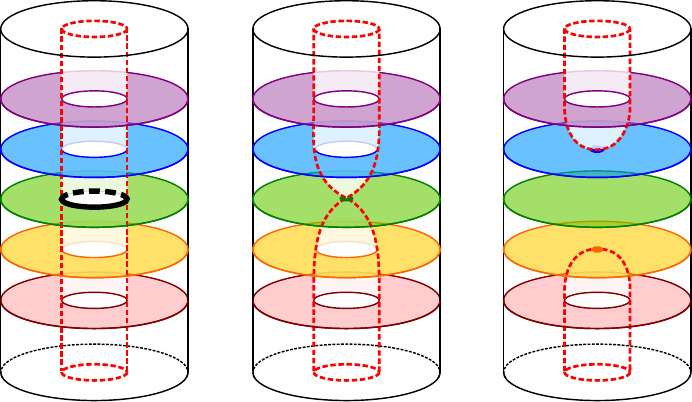
\caption{A neighborhood of an index two critical point of $h$. Each frame shows a neighborhood in the 3-manifold $Y_t$, which is doubled along the interior red boundary, and with colored surfaces representing fibers of $p_t$.  The height $t$ increases from left to right.}
\label{compress}
\end{figure}

The end result of this step, the singular fibration $p_{2+\eps}$ of $Y_{2+\eps}$, has $2n+2c$ singularities, of which $n+c$ are index one and $n+c$ are index two. All of these singularities are type II.  We extend further by letting $p_t$ have the same level sets as $p_{2 + 2\eps}$ for $t \in [2+\eps,3-4\eps]$, reparametrized so that all singularities are always type II.

\textbf{Step 3: Standardizing $p_t$ near the attaching spheres of 3-handles of $X$.}  
 (This step is essentially~\cite[Movie 20]{maggie}.) View the attaching spheres of the 3-handles as $c+1$ disjoint 2-spheres $S_1,\ldots, S_{c+1}$ in $Y_{3-4\eps}$. These spheres intersect the fibers of $p_{3-4\eps}$, inducing singular fibrations of each $S_i$ (after a small perturbation if necessary). For each $i$, if there are any saddle tangencies of $S_i$ with fibers of $p_{3-4\eps}$, then we may add pairs of canceling index zero and one or index two and three critical points to $p_t$ between $t = 3-4\eps$ and $t = 3-3\eps$, where we parametrize all cone singularities to be type II.  After this process is complete, we reposition the critical points of $p_t$ between $t=3-3\eps$ and $t=3-2\eps$ so that there are exactly two points in each $S_i$ that are tangent to fibers of $p_{3-2\eps}$. See Figure~\ref{simplify3handle}.


\begin{figure}[h!]
	\centering
	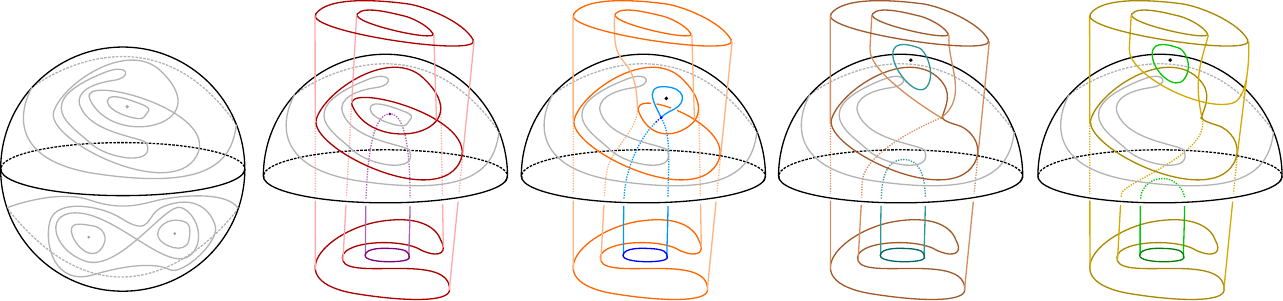
\caption{A movie of $p_t$ for $t \in [3-4\eps,3-2\eps]$, with four depictions of different $Y_t$, $t$ increasing. We draw $p_t$ near one hemisphere of $S_i$. } 
\label{simplify3handle}
\end{figure}

Between $t = 3-2\eps$ and $t = 3-\eps$, and for each attaching sphere $S_i$, we add an additional canceling pair of index zero and 1 critical points and a canceling pair of index two and 3 pairs of critical points to $p_t$ (taking cones to be type II), so that each $S_i$ contains one index one critical point and one index two critical point of $p_{3-\eps}$, and the other intersections of fibers of $p_{3-\eps}$ with $S_i$ are closed curves. 
See Figure~\ref{3handlecones}.
\begin{figure}[h!]
	\centering
	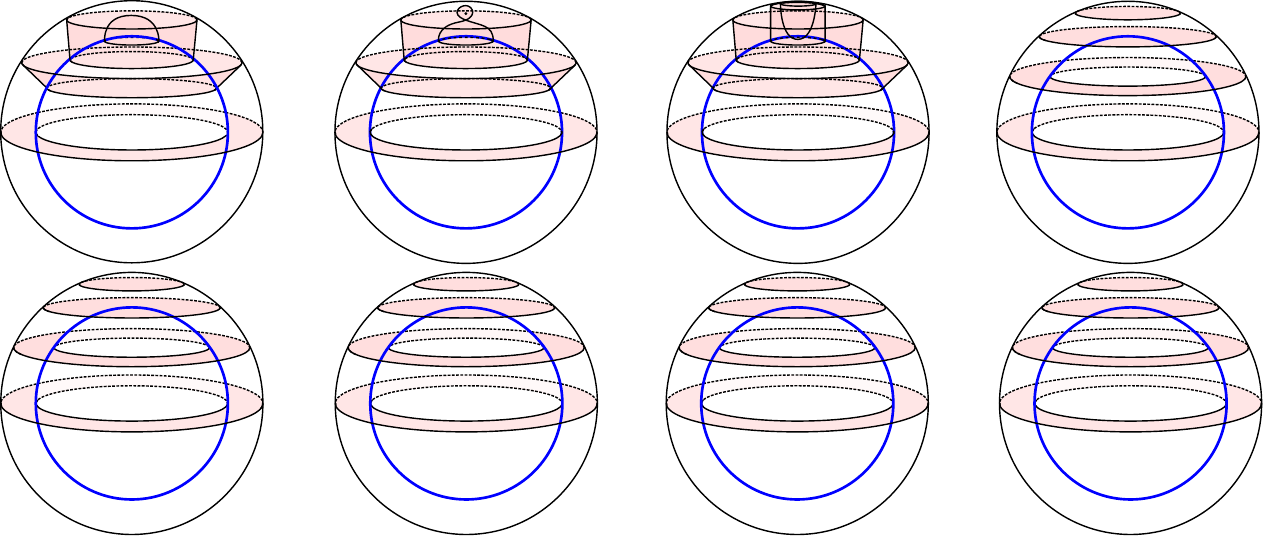
\caption{A movie of $p_t$ for $t \in [3-2\eps,3-\eps]$, with four depictions of $Y_t$, $t$ increasing. The two images in each column are glued along the blue sphere, which represents an attaching sphere $S_i$ of a 3-handle.} 
\label{3handlecones}
\end{figure}

\textbf{Step 4:  Extending $p_t$ across the 3-handles of $X$.}  
In step 3, we imposed a local model of $p_{3-\eps}$ in a neighborhood of the descending sphere of each $3$-handle of $X$. We can now extend $p_t$ over all of the $3$-handles, $t\in [3-\eps,3+\eps]$, by compressing each fiber intersecting the 3-handle attaching circle, as in Figure~\ref{3handle}. (This is~\cite[Movie 7: death movie 1]{maggie}.)  Note that extension across each 3-handle eliminates two critical points from $p_t$.  Thus, if $m$ is the total number of additional index zero and one or index two and three pairs added in step 3, then $m \geq 2(c+1)$ by construction, and $p_{3+\eps}$ has $2n + 2c - 2(c+1) + 2m = 2n - 2 + 2m$ critical points, of which $m$ are index zero or three, and the remaining $2n-2+m$ are index one or two.

\begin{figure}[h!]
	\centering
	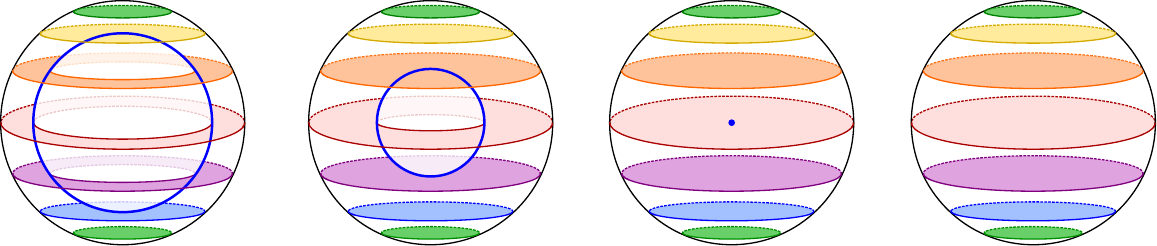
\caption{A movie of $p_t$ for $t \in [3-\eps,3+\eps]$, with four depictions of $Y_t$, $t$ increasing.  Each picture is doubled along the blue interior boundary, which represents the attaching sphere of a 3-handle.} 
\label{3handle}
\end{figure}

\textbf{Step 5: Canceling critical points of $p_t$.} Note $Y_{3+\eps}\cong S^3$. 
At this point, the singular fibration $p_{3+\eps}$ has critical points of each index.  Suppose that $p_{3+\eps}$ has more than one index zero critical point.  (We will eliminate extra index zero critical points using essentially the proof of~\cite[Lemma 4.2]{maggie}.) Without loss of generality, say $F_0$ is a fiber of $p_{3+\eps}$ containing all index zero critical points of $p_{3+\eps}$. Fix a gradient-like vector field $\nabla$ for $p_{3+\eps}$. Then there is an arc $\gamma$ between two index zero points $q_1,q_2\in F_0$ so that $\gamma$ is parallel to $\nabla$ and intersects exactly one critical point $r$ (which must be index one) of $p_{3+\eps}$ in its interior. The map $p_{3+\eps}:\gamma\to S^1=[0,2\pi]/\sim$ gives a natural distance function $d_{\gamma}$ on $\gamma$. Suppose $d_\gamma(q_1,r)\le d_\gamma(r,q_2)$.  Then we may extend $p_t$ to $t\in[3+\eps,3+2\eps]$ while canceling $q_1$ and $r$ along $\gamma$. (We use~\cite[Movie 11]{maggie} to move the cancelled critical points to be close together and~\cite[Movie 1]{maggie} to remove the two critical points.) 

Repeat this step and its dual for index three critical points as necessary for $t\in[3+2\eps,3+3\eps]$ until $p_{3+3\eps}$ has exactly one index zero and one index three critical point.  In the process, we have performed $m-2$ cancellations, so that $p_{3+3\eps}$ has $n$ index one and $n$ index two critical points remaining.



Since $Y_{3+3\eps}$ is $S^3$ and $S^3$ is simply-connected, we can lift $p_{3+3\eps}$ to obtain a Morse function $\hat p_{3+3\eps}:Y_{3+3\eps} \rightarrow \R$.  In this case, the image of $\hat p_{3+3\eps}(Y_{3+3\eps})$ is the interval $[a,b]$, with the index zero critical point occurring at height $a$ and the index three critical point occurring at height $b$.  By composing the straight-line homotopy from $[a,b]$ to $[0,\pi]$ and the covering map $\R \rightarrow S^1$ between $t = 3+3\eps$ and $3 + 4\eps$, we may assume that the image of $p_{3+4\eps}$ is $[0,\pi]$, so that $p_{3+4\eps}$ is an interval-valued Morse function.


Now extend $p_t$ across $t\in[3+4\eps,3+5\eps]$ while exchanging heights of critical points (via e.g.~\cite[Movie 11]{maggie}) so that $p_{3+5\eps}$ is self-indexing. Let $\Sigma$ be a thick surface for $p_{3+5\eps}$.  By choosing $\alpha$ and $\beta$ curves on $\Sigma$ that bound disks intersecting index one or index two critical points (respectively), we obtain a Heegaard diagram for $S^3$. See Figure~\ref{handleslide}. Extend $p_t$ across $t\in[3+5\eps,3+6\eps]$ while exchanging the heights of critical points to achieve handleslides of this Heegaard diagram, as in Figure~\ref{handleslide}. By Theorem~\ref{wald}, we may then take the Heegaard splitting of $S^3$ induced by $p_{3+6\eps}$ to be standard.

\begin{figure}[h!]
	\centering
	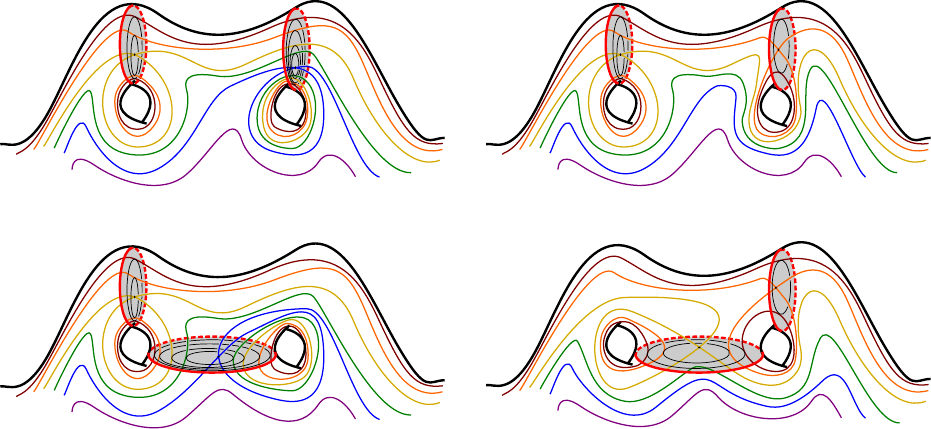
\caption{{\bf{Top:}} The surface $\Sigma$ and some index one critical points of $p_{3+5\eps}$ on one side of $\Sigma$. We draw $\alpha$ curves on $\Sigma$, which each bound a disk intersecting one index one critical point. As $t$ increases, we may exchange the heights of the index one critical points to achieve handle slides of the $\alpha$ curves.}
\label{handleslide}
\end{figure}

Now extend $p_t$ across $t\in[3+6\eps,3+7\eps]$. For each geometrically dual $\alpha$ and $\beta$ curve, during this time interval we cancel the corresponding pair of index one and index two critical points.  By construction, the only critical points of $\Ph$ occur when a critical point of some $p_t$ changes type, or when a pair of index one and index two critical points of the same type are canceled.  In the latter case, each cancellation introduces an index two critical point into the extension $\Ph$, and thus we conclude $\Ph$ has precisely $n$ of these.  See Figure~\ref{cancelcones}. 

\begin{figure}[h!]
	\centering
	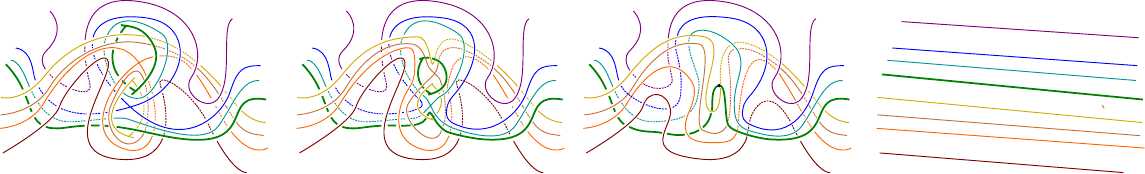
\caption{Left to right: we cancel an index one and an index two critical point of $p_t$ as $t$ increases from $3+6\eps$ to $3+7\eps$. The cancelled critical points are opposite index but both type II, so this introduces a singularity into the extension $\Ph:X\to S^1$. In this picture, the green fiber of $\Ph$ is singular.}
\label{cancelcones}
\end{figure}

Finally, $p_{3+7\eps}$ has an index zero critical point, an index three critical point, and no index one or 2 critical points. 
We then extend $p_t$ across the 4-handle of $X$ by attaching a 3-ball to each sphere fiber of $p_{3+6\eps}$. Thus, we obtain the extension $\Ph:X\to S^1$.

By construction, $\Ph$ has exactly $n$ interior critical points of index two and $2n$ type II boundary critical points.  Moreover, for all $t$, the index one and index two critical points of $p_t$ are type II (compressing), and thus the regular fibers of $\Ph$ are obtained from the regular fibers of $p$ by attaching $2$- and $3$-handles. That is, the regular fibers of $\Ph$ are handlebodies, as desired.
\end{proof}

Theorem~\ref{fibration2} follows immediately from a combination of the statements in Lemma~\ref{onedirection} and Proposition~\ref{otherdirection}.

\begin{remark}\label{thinremark}
Note that in the proof of Proposition~\ref{otherdirection}, if each 2-handle attaching circle of $X$ lies in some fiber of $p$ (with surface framing equal to its framing), then we may skip step 1. (In fact, it is not difficult to use \cite[Movie 19]{maggie} to amend step 2 to apply if the attaching circles of $X$ are each tangent to fibers of $p$ exactly twice and transverse otherwise.) Measuring the complexity of a singular fibration by the genera of its regular surfaces, the relaxed definition may give rise to simpler singular fibrations admitting extensions as compared to those admitting extensions to a generic map $P:X \rightarrow S^1 \X I$ as in Section~\ref{sec:extend}.  In particular, as noted above, the singular fibration of $\pd X$ in Proposition~\ref{otherdirection} need not be self-indexing for the construction to work.  On the other hand, given extensions $P:X \rightarrow S^1 \X I$ and $\Ph:X \rightarrow S^1$ for the same 4-manifold $X$, the extension $P$ induces a handle decomposition via Lemma~\ref{reverse_direction} with fewer handles than the one produced by $\Ph$ using Lemma~\ref{onedirection}.
\end{remark}

\begin{remark}\label{noextend}
One might wonder whether it could possible to strengthen Theorem~\ref{extensions} to show that if $K \subset S^3$ is handle-ribbon in a homotopy 4-ball, then \emph{every} singular fibration of $S^3_0(K)$ extends over handlebodies.  Unfortunately, this is not the case.  Using either definition of extension over handlebodies, if every singular fibration of $S^3_0(K)$ extends over handlebodies, then every Seifert surface $F$ for $K$ has a handle-ribbon derivative:  Given $F$, we can construct a singular fibration $p$ with $\wh F$ as a regular fiber.  If $p$ extends over handlebodies, then $\wh F$ bounds a handlebody $H$ in a 4-manifold $X$, where the union of the relative trace $B_0(K)$ and $X$ is a homotopy 4-ball $B$.  Any link $L \subset \wh F$ bounding disks cutting $H$ into a 3-ball can be isotoped into $F \subset S^3$, and by Lemma~\ref{reverse_direction} or Lemma~\ref{onedirection}, we have that $L$ is handle-ribbon in $B$.  However, as mentioned in the introduction, Cochran-Davis showed that there exists a ribbon knot $K$ with a genus one Seifert surface $F$ such that $F$ contains no slice derivative~\cite{CD}; thus, no singular fibration containing $\wh F$ as a fiber can extend over handlebodies (in either sense).
\end{remark}

\section{Natural trisections of singular handlebody extensions}\label{sec:trisections}

In this section, we study decompositions of the homotopy 4-spheres appearing in the context of extending a singular fibration to a generic map $P:X \rightarrow S^1 \X I$.  Recall that an R-link $L$ can be viewed as a Kirby diagram for a homotopy 4-sphere $X_L$, or a relative Kirby diagram for a homotopy 4-ball $B_L$.  By Theorem~\ref{hr}, every knot $K$ that is handle-ribbon in a homotopy 4-ball $B$ has an R-link derivative $L$ such that $B = B_L$.  In addition, by Theorem~\ref{fibration1}, there exists a self-indexing singular fibration $p:S^3_0(K) \rightarrow S^1$ extending to $P:X \rightarrow S^1 \X I$, where $B_L = B_P(K)$.  We let $X_P(K)$ denote the homotopy 4-sphere obtained by capping off $B_P(K)$ with a 4-ball, so that in this case we have $X_L = X_P(K)$.  The aim of this section is to show that the homotopy 4-sphere $X_L = X_P(K)$ admits a natural 4-manifold trisection.  One motivation for this investigation is to better understand the stable generalized Property R conjecture (the stable GPRC), which asserts that every R-link is stably equivalent to an unlink.  In~\cite{MZDehn}, Meier and the second author formulated an equivalent characterization of the stable GPRC via trisections.

Trisections of smooth 4-manifolds were introduced by Gay and Kirby in~\cite{GK}, in which they showed that every 4-manifold $X$ admits a trisection $\mathcal{T}$, a decomposition $X = X_1 \cup X_2 \cup X_3$, where $X_i$ is a 4-dimensional 1-handlebody and each intersection $X_i \cap X_j$ is a 3-dimensional handlebody.  These criteria imply that the triple intersection $X_1 \cap X_2 \cap X_3$ is a closed surface $\Sigma$, which we call the \emph{central surface}.  The complexity of the trisection is encoded in the parameters $g = g(\Sigma)$ and $k_i = \text{rk}(\pi_1(X_i))$; we call $\mathcal{T}$ a $(g;k_1,k_2,k_3)$-trisection.  If $k_1 = k_2 = k_3$, then $\mathcal{T}$ is said to be \emph{balanced}; otherwise, it is \emph{unbalanced}.  All trisections in this paper will be unbalanced.

The results here extend work of Meier and the second author, which we will obtain as a special case.

\begin{proposition}\cite[Proposition 9.1]{FHRkS}\label{fib}
Suppose $K$ is a fibered homotopy-ribbon knot in $S^3$ with genus $g$ fiber and fibration $p:S^3_0(K) \rightarrow S^1$.  Then for any handlebody extension $P:X \rightarrow S^1 \X I$ of $p$, the corresponding homotopy 4-sphere $X_P(K)$ admits a $(2g;0,g,g)$-trisection.
\end{proposition}

For the proof, we invoke the framework from~\cite{MZDehn}.  Suppose $L \subset S^3$ is an R-link.  We say that a Heegaard surface $\Sigma$ for $S^3$ is \emph{admissible} with respect to $L$ if $L$ is isotopic into a core of one of the handlebodies $H$ cut out by $\Sigma$, so that $H \setminus L$ is a compression body.  Equivalently, $\Sigma$ is admissible if the $n$-component link $L = \{L_i\}$ is isotopic into $\Sigma$ and there exist a collection of $n$ compressing disks $\{D_i\}$ such that $|L_i \cap D_j| = \delta_{ij}$.  We will use

\begin{lemma}\cite[Lemma 4]{MZDehn}\label{admissible}
Suppose $L \subset S^3$ is an $n$-component R-link with admissible genus $g$ Heegaard surface $\Sigma$.  Then $X_L$ admits a $(g;0,n,g-n)$-trisection.
\end{lemma}

For the next proof, it is important to note for any knot $K$ and singular fibration $p:S^3_0(K) \rightarrow S^1$, there is a singular open book decomposition $p':S^3 - K \rightarrow S^1$ induced by removing the dual $K^*$ from $S^3_0(K)$.  In this case, every nonsingular fiber of $p$ corresponds a page in the singular open book $p'$; that is, a Seifert surface for $K$.

\begin{proposition}\label{trisect}
Suppose $K$ is a handle-ribbon knot in $S^3$ with self-indexing singular fibration $p:S^3_0(K)$, whose thin and thick surfaces have genus $k$ and $\ell$, respectively.  For any handlebody extension of $P:X \rightarrow S^1 \X I$ of $p$, the corresponding homotopy 4-sphere $X_P(K)$ admits a $(k+\ell;0,k,\ell)$-trisection.
\end{proposition}

\begin{proof}
Suppose that $P:X \rightarrow S^1 \X I$ extends $p$ over handlebodies, and let $\wh F$ and $\wh G$ denote the thin and thick surfaces of $p$, respectively.  By the proof of Theorem~\ref{fibration1}, there exists an R-link derivative $L \subset \wh F$, where $X_L = X_P(K)$.  Let. $p':S^3 - K \rightarrow S^1$ be the singular open book decomposition induced by $p$, discussed above, and let $F$ and $G$ be the Seifert surfaces corresponding to $\wh F$ and $\wh G$.  We may assume that $F \cap G = K$ and $L \subset F$.

We claim that $\Sigma = F \cup G$ is a Heegaard surface for $S^3$ that is admissible with respect to $L$:  Since $p$ is self-indexing, $\Sigma$ cuts $S^3$ into two components, each of which is diffeomorphic to $(F \X I) \cup (\text{1-handles})$.  Since $F$ has nonempty boundary, $F \X I$ is a handlebody, and thus so is $(F \X I) \cup (\text{1-handles})$, so that $\Sigma$ is a Heegaard surface for $S^3$.  To see that $\Sigma$ is admissible, let $k = g(F)$, so that $L = \{L_i\}$ is a $k$-component link, and choose $k$ pairwise disjoint properly embedded arcs $\{a_i\}$ in $F$ such that $|L_i \cap a_j| = \delta_{ij}$.  Then each arc $a_j$ gives rise to a compressing disk $D_j = a_j \X I \subset F \X I$, and after a small isotopy of $a_j$, this disk can be chosen to be disjoint from the feet of the 1-handles, so that $D_j$ is also a compressing disk for $(F \X I) \cup (\text{1-handles})$.  We conclude that $\Sigma$ is admissible.

Finally, if $\ell = g(G)$, Lemma~\ref{admissible} implies that $X_P(K) = X_L$ admits a $(k+\ell;0,k,\ell)$-trisection as desired.
\end{proof}

Note that in the case that $K$ is a fibered knot, Proposition~\ref{fib} is a special case of Proposition~\ref{trisect}, since $g(G) =  g(F)$ in the fibered case. 

\begin{remark}\label{trisectionfamilyremark}
It is also worth noting that the components of the trisection $\mathcal{T}$ given in Proposition~\ref{trisect} arise naturally using a base diagram $P(Z_P)$.  Suppose that $K$ is handle-ribbon in $S^3$, with singular fibration $p:S^3_0(K) \rightarrow S^1$ extending over handlebodies to $P:X \rightarrow S^1\X I$.  As stated above, the first coordinate function $P_1:X \rightarrow S^1$ is a circular Morse function whose embedded fibers are handlebodies.  Let $D$ be the corresponding handle-ribbon disk for $K$ in $B_P(K)$, so that $X = B_P(K) \setminus D$.  As in the 3-dimensional case, $P_1$ induces a singular open book decomposition $P_1':B_P(K) - D$, where the restriction of $P_1'$ to $S^3 - K$ is the singular open book $p':S^3 - K \rightarrow S^1$, and each regular fiber of $P_1'$ is an embedded handlebody, whose boundary is a fiber of $p'$ capped off with the disk $D$.

Let $X_1$ be the 4-ball attached to $B_P(K)$ to get $X_P(K)$.  We will define $X_2$ and $X_3$ so that the following holds:
\begin{enumerate}
\item $X_1 \cap X_2 = (p')^{-1}([0,\pi])$,
\item $X_1 \cap X_3 = (p')^{-1}([\pi,2\pi])$, and
\item $X_1 \cap X_2 \cap X_3 = (p')^{-1}(\{0\}) \cup (p')^{-1}(\{\pi\})$,
\ee
so that as in the proof of Proposition~\ref{trisect}, the central surface is the connected sum of the thin surface (of genus $k$) and thick surface (of genus $\ell$) of $p$, viewed as having $K$ as a common boundary component.  For the moment, we have left $X_2 \cap X_3$ undefined.  This intersection should be the boundary connected sum of two handlebody fibers, call them $H_0$ and $H_{\pi}$, of the singular open book $P_1':B_P(K) - D \rightarrow S^1$.  It is clear that $H_0$ must be $(P_1')^{-1}(0)$.  Less clear is the handlebody $H_{\pi}$, since all of the interior critical points of $P_1$ occur at the critical value $\pi$; thus, we have a choice to make (in fact, we have $2^{\ell-k}$ such choices, as we will see below).

Note that $X_- = (P_1')^{-1}([0,\pi-\eps])$ and $X_+ = (P_1')^{-1}([\pi+\eps,2\pi])$ are 4-dimensional 1-handlebodies satisfying $\text{rk}(\pi_1(X_{\pm})) = \ell$.  See Figure~\ref{sections}.  Clearly, we want $X_- \subset X_2$ and $X_+ \subset X_3$, but the question is where to put the $\ell - k$ cusps contained in the portion of the base diagram $P(Z_P)$ contained in $[\pi - \eps,\pi+\eps] \X I$.  Let $X_{\pi} = (P_1')^{-1}([\pi-\eps,\pi+\eps])$.  Here, the cusps can be viewed as corresponding to 4-dimensional 2-handles, since $X_\pi$ can be obtained by attaching $\ell - k$ 2-handles to a collar neighborhood of $(P_1')^{-1}(\pi-\eps)$.  In the construction in Proposition~\ref{trisect}, we have $X_2 = X_- \cup X_{\pi}$, and the $\ell - k$ 2-handles cancel $\ell - k$ 1-handles of $X_-$, yielding the 1-handlebody $X_2$ with $\text{rk}(\pi_1(X_2)) = k$, as in the conclusion of the proposition.

However, in a dual construction we could add $X_{\pi}$ to $X_+$, or we could even break $X_{\pi}$ into its constituent 2-handle pieces, attaching some of the 2-handles to $X_-$ to obtain $X_2$ and some to $X_+$ so obtain $X_3$.  This choice is equivalent to choosing an arc $\gamma \subset S^1 \X I$ connecting $\{\pi\} \X \pd I$, transverse to the $I$ coordinate of $S^1 \X I$, and such that $\gamma$ avoids the cusp points of $P(Z_P)$.  Then $\gamma$ partitions the cusp points of $P(Z_P)$ into ``upper" and ``lower" cusp points, so that up to isotopy, where are $2^{\ell - k}$ choices for $\gamma$.  See Figure~\ref{trisectionfamily} for an illustration. For a particular choice, we let $H_{\pi} = P^{-1}(\gamma)$, so that $H_{\pi}$ is an embedded handlebody, and setting $X_2 \cap X_3 = H_0 \cup H_{\pi}$ as above determines a $(k+\ell; 0,k+n,\ell-n)$-trisection $\Tt_{\gamma}$ of $X_P(K)$, where $n$ is the number of lower cusp points determined by the arc $\gamma$, with $0 \leq n \leq \ell-k$.  The trisection produced by Proposition~\ref{trisect} agrees with $\Tt_{\gamma}$ when $n=0$.

\begin{figure}[h!]
	\centering
	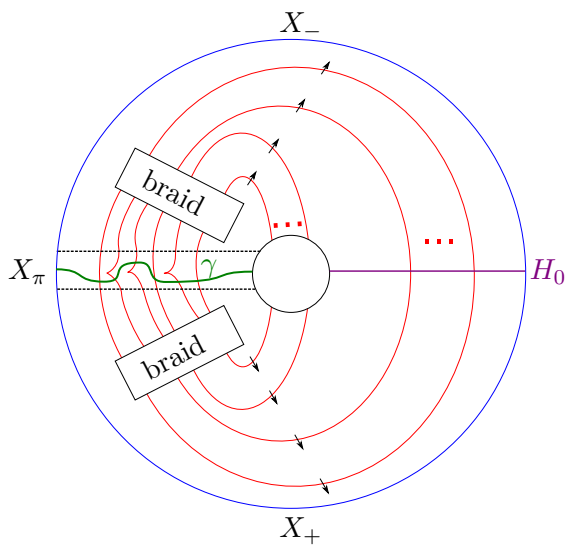
\caption{In Remark~\ref{trisectionfamilyremark}, the choice of the arc $\gamma$ determines the trisection $\Tt_{\gamma}$. The cusps just above $\gamma$ contribute 2-handles to $X_2$ while the cusps just below $\gamma$ contribute 2-handles to $X_3$.}
	\label{trisectionfamily}
\end{figure}

\end{remark}

\begin{question}
What is the relationship between different elements of the family $\{\Tt_{\gamma}\}$ of trisections of $X_P(K)$?  Is it possible that two of these trisections that differ by a single cusp are related by a single (unbalanced) stabilization and destabilization?
\end{question}

For more detailed constructions (including trisection diagrams) in the fibered case and further connections to the stable GPRC, the reader is encouraged to see Section 9 of~\cite{FHRkS}.

\bibliographystyle{amsalpha}
\bibliography{handle}

\end{document}